\documentclass[11pt,a4paper]{article}

\usepackage{tikz}
\usepackage{amsthm}
\usepackage{amsmath}
\usepackage{amssymb}
\usepackage{mathrsfs}
\usepackage{geometry}
\usepackage{graphicx}
\usepackage{amsfonts}
\usepackage{epstopdf}
\usepackage{enumerate}
\usepackage{enumitem}
\usepackage{hyperref}
\usepackage{caption}
\usepackage{soul}
\usepackage{subfig}
\usepackage[normalem]{ulem}

\numberwithin{equation}{section}

\allowdisplaybreaks

\newcommand{\Rmnum}[1]{\uppercase\expandafter{\romannumeral#1}} 

\def\Xint#1{\mathchoice
	{\XXint\displaystyle\textstyle{#1}}%
	{\XXint\textstyle\scriptstyle{#1}}%
	{\XXint\scriptstyle\scriptscriptstyle{#1}}%
	{\XXint\scriptscriptstyle\scriptscriptstyle{#1}}%
	\!\int}
\def\XXint#1#2#3{{\setbox0=\hbox{$#1{#2#3}{\int}$ }
		\vcenter{\hbox{$#2#3$ }}\kern-.6\wd0}}

\def\dashint{\Xint-}

\theoremstyle{plain}
\newtheorem{theorem}{Theorem}[section]
\newtheorem{proposition}[theorem]{Proposition}
\newtheorem{lemma}[theorem]{Lemma}
\newtheorem{corollary}[theorem]{Corollary}

\newtheorem{example}[theorem]{Example}

\theoremstyle{definition}
\newtheorem{remark}[theorem]{Remark}

\renewcommand{\thefootnote}{}

\AtEndDocument{
	\bigskip{\footnotesize%
		\textsc{School of Mathematics, Nanjing University, Nanjing 210093, China.}
		\par 
		\textsc{School of Mathematics, Hangzhou Normal University, Hangzhou 310036, China}  \par  
		\textit{E-mail address}: \texttt{gaojin@hznu.edu.cn} \par
	}
	\bigskip{\footnotesize%
		\textsc{Department of Mathematical Sciences, Tsinghua University, Beijing 100084, China}  \par  
		\textit{E-mail address}: \texttt{songyj24@mails.tsinghua.edu.cn} \par
	}
}

\begin{document}
\title{The H\"older regularity of harmonic function on bounded and unbounded p.c.f self-similar sets}
\author{Jin Gao and Yijun Song}
\date{}
\maketitle
\begin{abstract}
In this paper, we prove a generalized reverse H\"older inequality of harmonic functions on cable systems induced by post-critically finite (p.c.f.) self-similar sets. Furthermore, we also establish the H\"older regularity of harmonic functions on both bounded and unbounded p.c.f. self-similar sets, which does not involve
heat kernel estimates and resistance estimates.
\end{abstract}

\footnote{\textsl{Date}: \today}
\footnote{Jin Gao was supported by National Natural Science Foundation of China (Grant. No. 12271282), and Zhejiang Provincial Natural Science Foundation of China (Grant. No. LQN25A010019). The authors acknowledge Jiaxin Hu for reading the preliminary draft and Meng Yang for some helpful comments.}

\renewcommand{\thefootnote}{\arabic{footnote}}
\setcounter{footnote}{0}

 \section{Introduction}
\label{sec:introduction}

Recall the following classical result established independently by Grigor'yan \cite{Gri92} and Saloff-Coste \cite{Sal92,Sal95}: on a complete non-compact Riemannian manifold, the conjunction of the volume doubling condition and the Poincar\'{e} inequality is equivalent to the two-sided Gaussian heat kernel estimates given by
\begin{equation}\label{eqn_HK2}\tag*{HK(2)}
	\frac{C_1}{V(x,\sqrt{t})}\exp\left(-C_2\frac{d(x,y)^2}{t}\right)\le p_t(x,y)\le\frac{C_3}{V(x,\sqrt{t})}\exp\left(-C_4\frac{d(x,y)^2}{t}\right),
\end{equation}
where $d(x, y)$ denotes the geodesic distance and $V(x, r)=m(B(x, r))$ is the Riemannian measure of the open ball.

Note that the specific gradient upper bound
$$
|\nabla_y p_t(x, y)| \le \frac{C_1}{\sqrt{t}\, V(x, \sqrt{t})} \exp\left(-C_2\frac{d(x, y)^2}{t}\right),
$$
which matches \ref{eqn_HK2}, is known to hold only in particular settings, including Riemannian manifolds with non-negative Ricci curvature \cite{LY86}, Lie groups with polynomial volume growth \cite{Sal90} and the references therein.
Such pointwise gradient estimates are crucial for analyzing the $L^p$-boundedness of the Riesz transform $\nabla \Delta^{-1/2}$ for $ 1<p<+\infty$. The problem of characterizing the equivalence of the two semi-norms $\|\nabla f\|_p$ and $\|\Delta^{1/2} f\|_p$ on $C^\infty_0(M)$ for a non-compact Riemannian manifold $M$ was posed by Strichartz in \cite{Str83}.
Under volume doubling and  heat kernel upper bound,  Coulhon and Duong proved that the Riesz transform is $L^p$-bounded for all $p \in (1, 2)$ (see \cite[Theorem 1.1]{CD99}). For the case $p > 2$, Auscher, Coulhon, Duong, and Hofmann \cite[Theorem 1.3]{ACDH04} established a suitable $L^p$-estimate for the heat kernel gradient is equivalent to the $L^p$-boundedness of the Riesz transform under \ref{eqn_HK2}.
Further related results on the Riesz transform can be found in \cite{CD03,CCH06,CJKS20,Dev15} and the references therein.

In a non-compact doubling Dirichlet metric measure space with a ``carr\'e du champ'', Coulhon, Jiang, Koskela and Sikora \cite[Theorem 1.2]{CJKS20} established that assuming an upper Gaussian heat kernel bound and a local $L^\infty$-Poincar\'e inequality, an upper bound for the heat kernel gradient is equivalent to the following quantitative reverse $L^\infty$-H\"older inequality for gradients of harmonic functions: there exists $C_H \in (0, +\infty)$ such that for every ball $B = B(x_0, r)$ and every function $u$ harmonic in $2B$,
$$
\bigl\|\lvert \nabla u \rvert\bigr\|_{L^{\infty}(B)} \le \frac{C_H}{r} \dashint_{2B} \lvert u \rvert \, dm. 
$$
See also \cite{Jiang21} for related results. 

On many fractals, such as the Sierpi\'nski gasket and carpet, the heat kernel satisfies the two-sided sub-Gaussian estimates
\begin{equation}\label{eqn_HK}\tag*{HK($\beta$)}
	\frac{C_1}{V(x,t^{1/\beta})}\exp\left(-C_2\left(\frac{d(x,y)}{t^{1/\beta}}\right)
	^{\frac{\beta}{\beta-1}}\right)\leq p_t(x,y)\leq\frac{C_3}{V(x,t^{1/\beta})}\exp\left(-C_4\left(\frac{d(x,y)}{t^{1/\beta}}\right)
	^{\frac{\beta}{\beta-1}}\right),
\end{equation}
where $\beta>2$ denotes the walk dimension. When $\beta=2$, 
\ref{eqn_HK} 
reduces to the classical Gaussian estimate \ref{eqn_HK2}; for general $\beta \geq2$, Barlow and Bass \cite{BB04} introduced the cutoff Sobolev inequality and proved that \ref{eqn_HK} is equivalent to the conjunction of volume doubling, the Poincar\'e inequality, and the cutoff Sobolev inequality.

A natural problem arising from \ref{eqn_HK} is to establish the corresponding gradient estimate 
\begin{equation}\label{eqn_GHK}\tag{GHK}
	|\nabla_y p_t(x,y)| \leq \frac{C_3}{t^{1-\alpha/\beta}V(x, t^{1/\beta})} \exp\left( -C_4 \left( \frac{d(x, y)}{t^{1/\beta}} \right)^{\frac{\beta}{\beta-1}} \right).
\end{equation}
Note that the usual gradient operator cannot be defined on fractals, we turn to consider the  
fractal-like cable systems, see \cite{Chen15,BC23,DRY23}.
Research on such gradient estimates is closely connected to the study of quasi-Riesz transforms. For instance, Chen \cite{Chen15} gave gradient estimates for sub-Gaussian heat kernels and their link to the $L^p$ boundedness of operators $\nabla e^{\Delta}\Delta^{-\epsilon}$ for $p\in (1,2]$. Subsequently, Chen, Coulhon, Feneuil and Russ \cite{CCFR17} established the complete range of validity for the Riesz and reverse Riesz transforms on Vicsek graphs. Following this line of work, Devyver and Russ \cite{DR26} considered quasi-Riesz transforms on Vicsek-like cable systems, and Feneuil\cite{Fen26} further extended these results to graphical Sierpi\'nski gaskets.

In \cite{DRY23}, Devyver, Russ and Yang established a pointwise gradient estimates for the heat kernel on fractal-like cable systems by 
generalized reverse H\"older inequality \eqref{eq_GRH}: there exists a constant $C>0$ such that for any harmonic function $u$ in a ball $2B$ of radius $r$, 
\begin{align}
	\left\Vert \left\vert \nabla u\right\vert\right\Vert_{L^{\infty}(B)} \leq C \, \frac{\Phi(r)}{\Psi(r)} \dashint_{2B} |u|  dm.
\end{align}
where the scaling function $\Phi(r)$ and $\Psi(r)$ were introduced by \eqref{psi}.
Very recently, under assumptions of slow volume regular condition and the heat kernel upper bound,  the first author and Yang proved that the heat kernel H\"older continuity with or without exponential terms is equivalent to the H\"older regularity \eqref{eq_HR}(Generalization of condition \eqref{eq_GRH}) on the unbounded  metric measure space with a strongly recurrent condition in 
\cite{GY26}. This equivalence implies the gradient estimate \eqref{eqn_GHK}.

In this paper, we mainly consider the two open problems as follows:
\begin{itemize}
	\item  Proving condition \eqref{eq_GRH} on cable system induced by a class of p.c.f. self-similar sets (See \cite[p15544]{DRY23}).
	\item  Proving condition \eqref{eq_HR} on the p.c.f. self-similar sets (See \cite[p4]{GY26}).
\end{itemize}
The key point to our proof lies in the harmonic extension of the p.c.f. self-similar sets, without involving heat kernel estimates or resistant estimates. Our method fails to prove condition \eqref{eq_GRH} and \eqref{eq_HR} on the Sierpi\'nski carpet, due to the lack of harmonic extension.

The paper is organized as follows. In Section~\ref{sec:main_result}, we recall the definition of Dirichlet form on p.c.f. self-similar set and introduce the Theorem \ref{thm_GRH} and Theorem \ref{corol_GRH}. In Section~\ref{sec:proof_grh}, we give the oscillation inequality (\ref{eq_OSC}) of harmonic function on p.c.f. self-similar set. In Section~\ref{holder}, we prove the  Theorem \ref{thm_GRH} and Theorem \ref{corol_GRH}. Finally, we give three examples to verify the generalized reverse H\"older inequality \eqref{eq_GRH} and H\"older regularity \eqref{eq_HR} in Section~\ref{sec:example}.

\textbf{Notation}: The letters $C$,$C_{i}$,$C'$,$c$,$c_i$
are universal positive constants depending only on $K$ which may vary at
each occurrence.

\section{Definitions and the Main Result}
\label{sec:main_result}

The analysis on p.c.f. self-similar sets was originally developed by Kigami. For
convenience of readers, we will briefly recall the constructions of Dirichlet
forms on p.c.f. fractals 
\cite{Kig01book}. 
Then we give the Theorem \ref{thm_GRH} and Theorem \ref{corol_GRH}.

For simplicity, we will focus on the self-similar sets in $\mathbb{R}^d$. 
Let $M \geq 2$ be an integer.
Let $\{F_i\}_{i=1}^{M}$ be an iterated function system (IFS) of $\rho$-similitudes on $\mathbb{R}^d$. That is, for each $i = 1, \dots, M$,
\begin{align}
F_i(x) = \rho x + a_i,
\end{align}
where $\rho \in (0, 1)$, $a_i \in \mathbb{R}^d$.  The associated self-similar set $K$ is the unique non-empty compact set in $\mathbb{R}^d$ satisfying
\begin{equation}
	\label{eq:self_similar_set}
	K = \bigcup_{i=1}^{M} F_i(K),
\end{equation}
We now define the associated symbolic space. Let $ S = \{1, \cdots, M\} $ be the alphabet. Let $ W_0 = \{0\} $, and for $m \geq 1$ let $ W_m $ be the set of words $ w = w_1 \cdots w_m $ of length $ m $ with $ w_i \in S$. Write $ |w| = m $ for the length of $ w $, and set $ W_{*}= \bigcup_{m \geq 0} W_m $. For $ w \in W_m $, we write $ F_w = F_{w_1} \circ \cdots \circ F_{w_m} $, and call $ F_wK $ an $ m $-cell of $ K $.

 Let $\Sigma$ be the set of infinite words $ \omega = \omega_1 \omega_2 \cdots $ with $ \omega_i \in S $, equipped with the product topology. Define a continuous surjection $ \pi : \Sigma \to K $ by
\begin{align*}
\{\pi(\omega)\} = \bigcap_{n \geq 1} F_{[\omega]_m} K,
\end{align*}
where $ [\omega]_m = \omega_1 \cdots \omega_m \in W_m $ for each $ m \geq 1$ and for $ \omega = \omega_1 \omega_2 \cdots $ in $\Sigma$.

Define the critical set $ \mathcal{C} $ and the post-critical set $\mathcal{P} $ for $ K $ by
\begin{align*}
\mathcal{C} = \pi^{-1} \left( \bigcup_{i \neq j} (F_i K \cap F_j K) \right), \quad \mathcal{P} = \bigcup_{m \geq 1} \tau^m(\mathcal{C}),
\end{align*}
where  $\tau : \Sigma \to \Sigma $ is the left shift map defined as $ \tau (\omega_1 \omega_2 \cdots ) = \omega_2 \omega_3 \cdots $. If $ \mathcal{P} $ is a finite set, we call $\{F_i\}_{i=1}^M $ a \textit{post-critically finite} (p.c.f.) IFS, and $ K $ a p.c.f. self-similar set.  
For any $w\in W_n$, define
\begin{align*}
V_0 = \pi(\mathcal{P}),\quad V_m = \bigcup_{i=1}^M F_i V_{m-1},\quad V_* = \bigcup_{m \geq 0} V_m,
\end{align*}
Let $V_0 = \{p_1, p_2, \dots, p_N\}$. The closure of $ V_* $ under the Euclidean metric $d$ is $ K $.

Without loss of generality, we can assume that $\text{diam} (K) = 1$, where ‘diam’ denotes the diameter, by an affine transformation on $\{a_i\}_{i=1}^m$. Throughout the paper, we always assume that a p.c.f. self-similar set is connected so we will omit ‘connected’.
Hence $K$ admits an $\mathcal{H}^{\alpha}$-measure denoted by $\mu$ that is $\alpha$-regular , namely, the measure of any metric ball $B(x, r) := \{y \in K : d(x, y) < r\}$ with $ 0< r \leq 1$ in $K$ satisfies that
\begin{align}
	C^{-1}r^{\alpha}\leq \mu(B(x, r))\leq Cr^{\alpha}, 
\end{align}
where $C\geq 1$ is a constant and $\alpha=\text{dim}_H(K)=-\frac{\log M}{\log \rho}$. 

Now we recall how to construct a Dirichlet form on a p.c.f. self-similar set $ K $ (also see\cite{HW06}).   

Let $ V_0 $ be the boundary as above, and let $ \ell(V_0) = \{ f: V_0 \to \mathbb{R} \} $ be the space of real functions on $ V_0 $. A matrix $ H $ on $ V_0 $ is called a \emph{Laplace matrix} (or simply a Laplace) if it satisfies:
\begin{itemize}
	\item $ H_{pq} = H_{qp} \geq 0 $ for any $ p \neq q \in V_0 $,
	\item $ \langle f, H g \rangle := \sum_{p \in V_0} f(p) \left( \sum_{q \in V_0} H_{pq} g(q) \right) \leq 0 $ for all $ f, g \in \ell(V_0) $,
	\item $ H f = 0 $ if and only if $ f $ is constant.
\end{itemize}

Given a Laplace matrix $ H $ on $ V_0 $ and a family of positive numbers $ \mathbf{r} = \{r_i\}_{i=1}^{M} $, called resistance weights, we define an energy form $ \mathcal{E}_m $ on $ V_m $ for $ m \geq 0 $ by:
\begin{align*}
	\mathcal{E}_0(f, g) &= -\langle f, H g \rangle, \\
	\mathcal{E}_m(f, g) &= \sum_{w \in W_m} r_w^{-1} \mathcal{E}_0(f \circ F_w, g \circ F_w) \quad (m \geq 1),
\end{align*}
where $ r_w = r_{w_1} r_{w_2} \cdots r_{w_m} $ for $ w = w_1 w_2 \cdots w_m $.

We say that $ K $ possesses a \emph{harmonic structure} if there exists a pair $ (H, \mathbf{r}) $ such that the following \emph{harmonic extension problem} is solvable for any $ f \in \ell(V_0) $:
	\begin{equation}
		\label{eq:variational_problem}
		\min \{ \mathcal{E}_1(g,g) : g|_{V_0} = f \} = \mathcal{E}_0(f,f). 
	\end{equation}
If, in addition, $ r_i < 1 $ for all $ 1 \leq i \leq M $, then the harmonic structure is said to be regular.
The existence and regularity of a harmonic structure on fractals are discussed in \cite{Kig93}.

Assuming a regular harmonic structure $ (H, \{r_i\}_{i=1}^{M}) $. For any $ f: V_* \to \mathbb{R} $,  define the sequence $ \{\mathcal{E}_m(f)\}_{m \geq 0} $ by
\begin{align}
	\mathcal{E}_m(f):=\mathcal{E}_m(f,f)=\frac{1}{2}\sum_{w\in W_m} \sum_{x,y\in V_w} H_{x,y}^{w} (f(x)-f(y))^2, 
\end{align}
where $H_{x,y}^w=r_w^{-1}H_{F_w^{-1}(x),F_w^{-1}(y)}$ with $x,y\in V_w=F_w(V_0)$. By \eqref{eq:variational_problem}, we know that $\mathcal{E}_m$ is non-decreasing. Define the Dirichlet form $ (\mathcal{E}, \mathcal{F}) $ on $ K $ by:
\begin{equation} \label{DF_Def}
\begin{cases}
	\mathcal{E}(f) :=\mathcal{E}(f,f)= \lim_{m \to \infty} \mathcal{E}_m(f,f), \\
	\mathcal{F} = \{ f \in C(K) : \mathcal{E}(f) < \infty \},
\end{cases}
\end{equation} 
It is known that for a suitable Borel measure $\mu$ on $K$ that charges every cell $K_w$ ($w \in W_n$), the form $(\mathcal{E}, \mathcal{F})$ is a local, regular, irreducible Dirichlet form on $L^2(K, \mu)$. Moreover, the energy $ \mathcal{E} $ is self-similar:
\begin{equation*}
	\label{eq:self_similar_energy}
	\mathcal{E}(f) = \sum_{i \in S} r_i^{-1} \mathcal{E}(f \circ F_i) \quad \text{for all } f \in \mathcal{F}.
\end{equation*}
Without loss of generality, we assume $r_i=r_j$ for all $i,j \in S$.
Under this assumption, we have that $r_i=\rho^{\beta-\alpha}$, where $\beta$ is the walk dimension associated with $(\mathcal{E}, \mathcal{F})$ and  it satisfies $2\leq \beta \leq \alpha + 1$.

Now, we first introduce the cable system $X$ induced by p.c.f. self-similar set $K$, also see \cite[Section 3]{DRY23}. 
Let $G_0$ be the set of all initial vertices of $X$ with $V_0\subseteq G_0$. Define 
\begin{align*}
	G_{k+1} = \cup_{i=1}^{M}F_i(G_k),~	G^{(k)} = \rho^{-k}G_k.
\end{align*}
Let $ G= \cup_{k=1}^{\infty}G^{(k)}, $ and $ E = \{ [p, q] : p, q \in G, |p - q| = 1 \} $, then $(G, E)$ is an infinite, locally bounded, connected graph, the corresponding unbounded cable system with $V=\cup_{k=1}^{\infty}\rho^{-k}V_k$. Each closed (open) cable is a closed (open) interval in $\mathbb{R}^2$ and
\begin{align}
	X = \bigcup_{\substack{p,q \in G \\ |p-q| = 1}} [p, q] \subseteq \mathbb{R}^2;
\end{align}
where $[p, q]$ denotes the closed interval with endpoints $ p, q \in \mathbb{R}^2 $. 
For any $n\ge0$, we say that a subset $W$ of $X$ is an $n$-skeleton if $W$ is a translation of the intersection of the closed convex hull of $G^{(n)}$ and $X$. 

Two functions associated functions $\Phi$ and $\Psi$ defined by
\begin{align}
\Phi(r)=r \mathbf{1}_{(0,1)}(r)+r^{\alpha}\mathbf{1}_{[1,+\infty)}(r),  \quad
\Psi(r)=r^2 \mathbf{1}_{(0,1)}(r)+r^{\beta}\mathbf{1}_{[1,+\infty)}(r).\label{psi}
\end{align}
Here, $\Phi$ typically governs the volume growth $V(x, r) \asymp \Phi(r)$, and $\Psi$ governs the time-scaling in the sub-Gaussian heat kernel estimates.

We say that the \emph{generalized reverse H\"older inequality} \eqref{eq_GRH} holds if there exists a constant $C \in (0, +\infty)$ such that for any ball $B = B(x_0, r)$ and for any function $u \in \mathcal{F}$ harmonic in $2B$, we have
\begin{align}
	\| |\nabla u| \|_{L^\infty(B)} \leq C \, \frac{\Phi(r)}{\Psi(r)} \dashint_{2B} |u| \, dm, \label{eq_GRH}\tag{GRH}
\end{align}
which is equivalent to an more explicit formulation, considering the piecewise definitions of $\Phi$ and $\Psi$, 
\begin{equation*}
	\label{eq:grh_explicit}
	\| |\nabla u| \|_{L^\infty(B)} \leq 
	\begin{cases}
		\dfrac{C}{r} \dashint_{2B} |u| \, dm, & \text{if } r \in (0, 1), \\
		\dfrac{C}{r^{\beta - \alpha}} \dashint_{2B} |u| \, dm, & \text{if } r \in [1, +\infty).
	\end{cases}
\end{equation*}

\begin{theorem}\label{thm_GRH}\rm
	Let $X$ be a cable system  induced by p.c.f. self-similar set $K$. We have that generalized reverse H\"older inequality \eqref{eq_GRH} holds on $X$.
\end{theorem}

Next, we will introduce the unbounded p.c.f. self-similar sets. the Let $K_0 = K$, and for $n \ge 1$ define $K_n = \rho^{-n} K$. Set $K_\infty = \bigcup_{n \ge 0} K_n$. We call $K_n(n < \infty)$ a \emph{bounded} p.c.f. self-similar set, and $K_\infty$ the corresponding unbounded p.c.f. self-similar set.
Define $V_\infty = \bigcup_{n \ge 0} \rho^{-n} V_*$. Then $V_\infty$ is a proper subset of $K_\infty$, and its closure in the Euclidean metric equals $K_\infty$.
Let $d_\infty$ be the geodesic metric (length metric) on $K_\infty$ induced by the Euclidean distance $d$. It is known that $(K_\infty, d_\infty)$ is a complete, separable, geodesic metric space.

In what follows, the symbol $K_n$ will be used with $n \in [0,\infty]$; thus $K_n$ may denote either a bounded set (when $n \in [0,\infty))$ ) or the unbounded set $K_\infty$ (when $n = \infty$).

We say that the H\"older regularity condition \eqref{eq_HR} holds if there exists $C > 0$ such that for any ball $B = B(x_0, r)$, for any $u \in \mathcal{F}$ harmonic in $2B$, and for $m$-a.e. $x, y \in B$ with $x \neq y$,
\begin{align}
 \frac{|u(x) - u(y)|}{d(x,y)^{\beta-\alpha}}\leq  \frac{C}{r^{\beta-\alpha}}\dashint_{2B} |u| \, dm.\label{eq_HR}\tag{HR}
\end{align}

\begin{theorem}\label{corol_GRH}\rm
	Let $K_n (0\leq n\leq\infty)$ be the bounded and unbounded p.c.f. self-similar set induced by by p.c.f. self-similar set $K$. We have that the H\"older regularity condition \eqref{eq_HR} holds on $K_n$.
\end{theorem}

The key point of proof is that we combine \cite[Appendix]{Tep00} by Teplyaev and \cite[Theorem 8.3]{Str00} by Strichartz to give the oscillation inequality (\ref{eq_OSC}), which is important to cover classical
Sierpi\'nski gasket and Vicsek set instead of \cite[Theorem 8.3]{Str00} or \cite[Example 5.6]{THP14} only for gasket-type fractals.

\begin{corollary}\rm
In an unbounded p.c.f. self-similar set, under volume regular condition  and upper heat kernel estimate, The H\"older regularity condition \eqref{eq_HR} implies
the heat kernel H\"older continuity with or without exponential terms (also see \cite[Theorem 2.1]{GY26}). 
\end{corollary}

\begin{remark}
The result \eqref{eq_HR} in unbounded p.c.f. self-similar set have proved in  \cite{GY26} by two-sided heat kernel estimates and resistance estimates, However, Theorem \ref{corol_GRH} is developed intrinsically from the harmonic structure.
\end{remark}

\begin{remark}
Actually, we focuses on `homogeneous' p.c.f. self-similar sets, where `homogeneous' means that constant contraction ratio $\rho$. A similar approach is adapted to establish the corresponding result for the inhomogeneous case.
\end{remark}

\section{The oscillation inequality of harmonic function}
\label{sec:proof_grh}
The main purpose of this section is to prove an oscillation inequality of harmonic function on the p.c.f self-similar set. 
The proof is adapted from \cite[Appendix]{Tep00} and \cite[Theorem 8.3]{Str00}. For completeness, we present the proof in this paper.

\begin{lemma}(\cite[Lemma A.1.]{Tep00})\rm
	\label{lemma:current-equality}
	If $u$ is a harmonic function on $V_m$, and $u$ restricted to $V_0$ is either $0$ or $1$, then the following inequality holds:
	\begin{equation}
		H_{x,y}^{w}|u(x)-u(y)|\leq \mathcal{E}_m(u,u).\label{A.3}
	\end{equation}
	for any $w \in W_m$, $x,y\in V_w$.
\end{lemma}

\begin{remark}
By \cite[Remark A.2]{Tep00}, we know that inequality \eqref{A.3} has a clear interpretation in terms of electrical networks. Consider a network in which nodes $ x $ and $ y $ are connected by a resistor of conductance $ H_{x, y}^w $. When each boundary point is held at electric potential $ 0 $ or $1$, the expression $\mathcal{E}_m(u, u) $ equals the total current flowing through the network since a formula $ E = I U $ (power equals to current times potential difference). Therefore inequality \eqref{A.3}  states that the current through any single resistor does not exceed the total network current.
\end{remark}

Before that, we give the proof of two necessary proposition as follows.

\begin{proposition}\rm
	\label{lemma:mvp}
	If $u$ is a harmonic function on $V_m$, then for any non-boundary point $x$, that is
	$x\in V_m \backslash V_0$, we have
	\begin{equation}
		\label{mvp}
		\sum_{w \in W_m:x\in V_w \backslash V_0} \sum_{ y \in V_w} H_{xy}^{w} (u(x)-u(y))=0
	\end{equation}
\end{proposition}

The behaviour of the currents in a network is described by Kirchhoff's current law, which states that if $x \in V_m \backslash V_0$, and no external current flows into or out of $x$ (see \cite[Lemma 2.3]{Bar17}).

\begin{proof}
	If $u$ is harmonic on $V_m$, it minimizes the energy functional $\mathcal{E}_m$ with the boundary condition $u|_{V_0}$. That is,
	\begin{align*}
		\mathcal{E}_m(u,u) &= \frac{1}{2} \sum_{w \in W_m} \sum_{x,y \in V_w} H_{xy}^{w} (u(x) - u(y))^2 \\
		&= \min \left\{ \mathcal{E}_m(f,f) : f \big|_{V_0} = u|_{V_0} \right\} \\
		&= \min \left\{ \frac{1}{2} \sum_{w \in W_m} \sum_{x,y \in V_w} H_{xy}^{w} (f(x) - f(y))^2 : f \big|_{V_0} = u|_{V_0} \right\}.
	\end{align*}
	
	Note that the minimum is attained at critical point, we have
	\begin{align*}
		\frac{\partial \mathcal{E}_m(f,f)}{\partial f(x)} \bigg|_{u|_{V_0}=f|_{V_0}} =
		2 \sum_{w \in W_m:x\in V_w \backslash V_0} \sum_{ y \in V_w} H_{xy}^{w} (f(x) - f(y)) = 0,
	\end{align*}
	since 
	\begin{align*}
		\frac{\partial \mathcal{E}_m(f,f)}{\partial f(x)} 
		&= \frac{1}{2} \sum_{w \in W_m:x\in V_w \backslash V_0} \sum_{ y \in V_w} \left[ H_{xy}^{w} \cdot 2(f(x)-f(y)) + H_{yx}^{w} \cdot 2(f(x)-f(y)) \right] \\
		&= 2 \sum_{w \in W_m:x\in V_w \backslash V_0} \sum_{ y \in V_w} H_{xy}^{w} (f(x)-f(y)),
	\end{align*}
	where we used the symmetry $H_{xy}^{w} = H_{yx}^{w}$.
\end{proof}

\begin{proposition}\rm
	\label{lemma:total_current}
	If $u$ is a harmonic function on $V_m$, 
	and $u$ restricted to $V_0$ is either $0$ or $1$, then we have
	\begin{equation}
		\label{equation:electric}
 	\sum_{ x \in V_0:u(x)=0}\sum_{w \in W_m:x\in V_w}\sum_{ y \in V_w} H_{x,y}^w u(y) = \mathcal{E}_m(u,u).
	\end{equation}
\end{proposition}

\begin{proof}
	Let $V_0 = \{p_1, p_2, \dots, p_N\}$ be the boundary vertices. Without loss of generality, 
	assume $u(p_1) = 1$ and $u(p_i) = 0$ for $i = 2, 3, \dots, N$.
	
	For any function $f \in l(V_m)$ with $f|_{V_0} = u|_{V_0}$, we have 

	\begin{align*}
			\mathcal{E}_m(f,f)=&	
			\frac{1}{2}\sum_{w \in W_m} \sum_{x,y \in V_w} H_{x,y}^w (f(x)-f(y))^2
			\\
			= &\sum_{w \in W_m:p_1\in V_w} \sum_{y \in V_w} H_{p_1,y}^w (1-f(y))^2 + \sum_{i=2}^N \sum_{w \in W_m:p_i\in V_w} \sum_{y \in V_w\setminus\{p_1\}} H_{p_i,y}^w (f(y))^2 \\+& \frac{1}{2}\sum_{w \in W_m} \sum_{x,y \in V_w\setminus V_0} H_{xy}^w (f(x)-f(y))^2
			\\
			=& \sum_{y \in V_m\setminus V_0} \left( \sum_{w \in W_m:y\in V_w,p_1\in V_w} H_{p_1,y}^w \right) f^2(y)+\sum_{y \in V_m\setminus V_0} \left( \sum_{i=2}^N  \sum_{w \in W_m:y\in V_w,p_i\in V_w} H_{p_i,y}^w \right) f^2(y)\\	+&  \sum_{x \in V_m \setminus V_0}  \sum_{w \in W_m:x\in V_w\setminus V_0}\sum_{ y \in V_w\setminus V_0} H_{x,y}^w f^2(x)
			- \sum_{x \in V_m\setminus V_0} \sum_{y \in V_m\setminus V_0}   \sum_{w \in W_m:x,y\in V_w} H_{x,y}^w f(x) f(y)\\-& 2 \sum_{y \in V_m\setminus V_0} \left(  \sum_{w \in W_m:y\in V_w,p_1\in V_w} H_{p_1,y}^w \right) f(y)+  \left( \sum_{w \in W_m:p_1\in V_w} \sum_{y \in V_w} H_{p_1,y}^w \right) \\	-& 2 \left(  \sum_{w \in W_m:p_1\in V_w } H_{p_1,p_1}^w \right) +  \sum_{w \in W_m:p_1\in V_w}   H_{p_1,p_1}^w\\	
			\\=& f_1^T B f_1 - 2 C ^Tf_1 + A.
	\end{align*}
	where $f_1$ is $f$ with the values on $V_0$ removed. Since $B$ is a diagonally dominant matric and positive definite, the minimizer $u_1'$ of $\mathcal{E}_m(f,f)$ satisfies:
	\begin{align*}
		Bu_1'=C.
	\end{align*}

	Let $u_0$ be $u_1'$ plus the values of $u$ on $V_0$. A straightforward calculation yields
	\begin{align*}
		\mathcal{E}_m(u_0,u_0)= \sum_{w \in W_m:p_1\in V_w} \sum_{y\in V_w} H_{p_1,y}^w (1 - u_0(y)).
	\end{align*}
Note that $u$ is harmonic on $V_m$, we have $u_0=u$ and
	\begin{align*}
		\mathcal{E}_m(u,u)= \sum_{w \in W_m:p_1\in V_w} \sum_{y\in V_w} H_{p_1,y}^w (1 - u(y)).
	\end{align*}

Similarly, if $u$ is harmonic on $V_m$ and $u$ restricted to $V_0$ is either $0$ or $1$, then
	\begin{align*}
	\sum_{ x \in V_0:u(x)=1}\sum_{w \in W_m:x\in V_w}\sum_{ y \in V_w} H_{x,y}^w  (1-u(y)) = \mathcal{E}_m(u,u).
	\end{align*}
By $\mathcal{E}_m(1-u,1-u)=\mathcal{E}_m(u,u)$, we have
	\begin{align*}
		\sum_{ x \in V_0:u(x)=0}\sum_{w \in W_m:x\in V_w}\sum_{ y \in V_w} H_{x,y}^w u(y) = \mathcal{E}_m(u,u),
	\end{align*}
which complete the proof.
\end{proof}
Now we give the proof of Lemma \ref{lemma:current-equality} as follows.
\begin{proof}[Proof of Lemma \ref{lemma:current-equality}]
	Define 
	\begin{align}
		E_a:=\{(x,y,w):w\in W_m,x,y \in V_w,u(x)< a,u(y)\geq a\}.
	\end{align}	
	
	We claim that for any $a\in (0,1]$,
   \begin{align}
   	\sum_{(x,y,w) \in E_a} H_{x,y}^w (u(y)-u(x))=\mathcal{E}_m(u,u).
   \end{align}
   Indeed, by Proposition~\ref{lemma:mvp}, we have
	\begin{align}
		 \sum_{(x,y,w) \in E_a} &H_{x,y}^w (u(y)-u(x)) 
		 =\sum_{w \in W_m} \sum_{\substack{x,y \in V_w \\ u(y) \geq a,u(x)< a}} H_{x,y}^w (u(y)-u(x))\nonumber\\
		 &= \sum_{w \in W_m} \sum_{\substack{x \in V_0\cap V_w, u(x)=0 \\ y \in V_w, u(y) \geq a}} H_{x,y}^w u(y)+ \sum_{w \in W_m} \sum_{\substack{x \in V_w\backslash V_0, 0<u(x)<a\\ y \in V_w, u(y)\geq a}} H_{x,y}^w (u(y)-u(x))\nonumber\\
		 &= \sum_{w \in W_m} \sum_{\substack{x \in V_0\cap V_w, u(x)=0 \\ y \in V_w, u(y) \geq a}} H_{x,y}^w u(y)+ \sum_{w \in W_m} \sum_{\substack{x \in V_w\backslash V_0, 0<u(x)<a\\ y \in V_w, u(y)\textless a}} H_{x,y}^w (u(x)-u(y)).\label{eq3.3.1}
	\end{align}
Based on the symmetry of $x$ and $y$ in $V_w \backslash V_0$, it follows that 
	\begin{align}
			&\sum_{w \in W_m} \sum_{\substack{x \in V_w\backslash V_0, 0<u(x)<a\\ y \in V_w, u(y)< a}} H_{x,y}^w (u(x)-u(y))\nonumber\\
			&=\sum_{w \in W_m} \sum_{\substack{x \in V_w\backslash V_0, 0<u(x)<a\\ y \in V_w \backslash V_0, 0<u(y)< a}} H_{x,y}^w (u(x)-u(y))+\sum_{w \in W_m} \sum_{\substack{x \in V_w\backslash V_0, 0<u(x)<a\\ y \in V_0\cap V_w, u(y)=0}} H_{x,y}^w (u(x)-u(y))\nonumber\\
			&=0+\sum_{w \in W_m} \sum_{\substack{x \in V_w\backslash V_0, 0<u(x)<a\\ y \in V_0\cap V_w, u(y)=0}} H_{x,y}^w u(x).\label{eq3.3.2}
	\end{align}

By Proposition \ref{lemma:total_current}, plugging \eqref{eq3.3.2} into \eqref{eq3.3.1},
we have
	\begin{align*}
	\sum_{(x,y,w) \in E_a} H_{x,y}^w (u(y)-u(x))=\sum_{w \in W_m} \sum_{\substack{x \in V_0, u(x)=0 \\ y \in V_w}} H_{x,y}^w u(y)= \mathcal{E}_m(u,u).
	\end{align*} 
	Since $u(x)< a$ and $u(y)\geq a$, we obtain 
	\begin{align*}
		H_{x,y}^{w}|u(x)-u(y)|\leq \mathcal{E}_m(u,u),
	\end{align*}
which complete the proof.
\end{proof}

   Next, we will consider the estimate of the local energy of harmonic functions. Based on \cite[Theorem 8.3]{Str00}, we obtain the 
   oscillation inequality of harmonic functions.

   We first provide an introduction to the harmonic function space.
   Define the space of harmonic functions by $\mathcal{H}\subset C(K)$, which is $N$-dimensional since any harmonic function is uniquely determined by its boundary values.
	We define the norm of $\mathcal{H}$ by $\|h\|_{\mathcal{H}}^2 = \mathcal{E}(h, h) + (\sum_{x \in V_0} h(x))^2$. 
 For every $i=1, \dots, N$,  the linear map $M_i: \mathcal{H} \to \mathcal{H}$ be defined by $M_i h = h \circ F_i$.
	Define $\widetilde{\mathcal{H}}:=\{(h_1,h_2,\dots,h_{N-1})\in \mathbb{R}^{N-1}:\sum_{j=1}^{N-1}h_j=0\}$.
	
	The following Proposition
	\ref{lem:osc} is adapted by \cite[Theorem 5]{Tep00}, which is a key to prove the oscillation inequality (\ref{eq_OSC}).

\begin{proposition}\rm
	\label{lem:osc}
	There exists a constant $C(V_0,H)$ such that for any harmonic function $u$ on $K$ and for any $w\in W_m$, we have
	\begin{equation}
		\label{equation:osc_energy}
		\mathcal{E}_0(M_w(u|_{V_0}))\leq Cr_w^2\mathcal{E}_0(u|_{V_0}).
	\end{equation}
\end{proposition}

\begin{proof}
Firstly, we prove the result for special cases: for any $i\in \{1,2,\dots,N\}$, assume $u_i$ is harmonic on $K$ with $u_i(p_i)=1$ and $u_i(p_j)=0$ for all $j\neq i$. 

For any $w\in W_m$, $x,y\in V_w$, let $x=F_w(p_k)$, $y=F_w(p_l)$, where $k,l\in \{1,2,\dots,N\}$. By Lemma~\ref{lemma:current-equality}, we have:
	\begin{align*}
		H_{x,y}^w|u_i(x)-u_i(y)| &= \frac{1}{r_w}H_{p_k,p_l}|u_i(F_w(p_k))-u_i(F_w(p_l))| \\
		&\leq \mathcal{E}_m(u_i|_{V_m},u_i|_{V_m}) = \mathcal{E}_0(u_i|_{V_0},u_i|_{V_0}),
	\end{align*}
	where we use the fact that the  $u_i$ is a harmonic function in the second line,
	it implies
	\begin{align*} 
		H_{p_k,p_l}|u_i(F_w(p_k))-u_i(F_w(p_l))|\leq r_w \mathcal{E}_0(u_i|_{V_0},u_i|_{V_0}).
	\end{align*}
	
	 On $\widetilde{\mathcal{H}}$, define $\|h\|_{\widetilde{\mathcal{H}}}^2=\mathcal{E}_0(h,h)$ and $\|h\|_{\widetilde{\mathcal{M}}}=\max_{1\leq m<j\leq N}H_{p_m,p_j}|h_m-h_j|$. Since $\mathcal{H}$ is finite-dimensional, $\|\cdot\|_{\widetilde{\mathcal{H}}}$ and $\|\cdot\|_{\widetilde{\mathcal{M}}}$ are equivalent norms. Thus, there exist constants $C_1:=C_1(V_0,H)>0$ and $C_2:=C_2(V_0,H)>0$ such that for any $h\in \widetilde{\mathcal{H}}$, we have $C_1\|h\|_{\widetilde{\mathcal{M}}}\leq \|h\|_{\widetilde{\mathcal{H}}}\leq C_2\|h\|_{\widetilde{\mathcal{M}}}$.
	
	For any harmonic function $u\in \mathcal{H}$ on $K$, define $\widetilde{u}=u-\frac{\sum_{i=1}^{N}u(p_i)}{N}$. For any $h\in\mathbb{R}^N$, when no confusion arises, similarly define $\widetilde{h}=(h_i-\frac{\sum_{j=1}^{N}h_j}{N})_{i=1,2,...,N}$. Then $\widetilde{u}|_{V_0},\widetilde{h}\in \widetilde{\mathcal{H}}$ and $\widetilde{u}$ is also a harmonic function, and $\sum_{i=1}^{N}\widetilde{u}(p_i)=0$, $\mathcal{E}_0(\widetilde{u}|_{V_0})=\mathcal{E}_0(u|_{V_0})$, and $\widetilde{u}(x)-\widetilde{u}(y)=u(x)-u(y)$.
	Therefore, we have
	\begin{align*}
		\sqrt{\mathcal{E}_0({M_w(u_i|_{V_0})})} &= \sqrt{\mathcal{E}_0(\widetilde{M_w(u_i|_{V_0})})} \\
		&\leq C_2 \|{\widetilde{M_w({u_i}|_{V_0})}}\| _{\widetilde{\mathcal{M}}}\\
		&= C_2 \max_{1\leq m<j\leq n} H_{p_m,p_j} \left|[\widetilde{M_w({u_i}|_{V_0})}]_m -[\widetilde {M_w({u_i}|_{V_0})}]_j\right| \\
		&= C_2 \max_{1\leq m<j\leq n} H_{p_m,p_j} \left|\left[{M_w(u_i|_{V_0})}\right]_m -\left[ {M_w(u_i|_{V_0})}\right]_j\right| \\
		&= C_2 \max_{1\leq m<j\leq n} H_{p_m,p_j} \left|u_i(F_w(p_m)) - u_i(F_w(p_j))\right| \\
		&\leq C_2 r_w \mathcal{E}_0(u_i|_{V_0}),
	\end{align*}
where $[\cdot]_j$ means the $j$th element.

	Let $A = \max_{1\leq i\leq N-1} \sqrt{\mathcal{E}_0(u_i|_{V_0})}$, we have
	\begin{align*}
		\mathcal{E}_0(M_w(u_i|_{V_0})) \leq (C_2)^2 r_w^2 A^2 \mathcal{E}_0(u_i|_{V_0}).
	\end{align*}
	
	Now we prove the general result. For any harmonic function $u\in \mathcal{H}$ on $K$, since $\widetilde{u}|_{V_0} \in {\widetilde{\mathcal{H}}}$ and $\{\widetilde{u_i}|_{V_0}: i=1,2,\dots,N-1\}$ forms a basis of ${\widetilde{\mathcal{H}}}$, there exist coefficients $(c_1,c_2,\dots,c_{N-1})$ such that $\widetilde{u}|_{V_0} = \sum_{i=1}^{N-1}c_i\widetilde{u_i}|_{V_0}$. By linearity of harmonic functions, we have $\widetilde{u} = \sum_{i=1}^{N-1}c_i\widetilde{u_i}$ and $M_w(\widetilde{u}|_{V_0}) = \sum_{i=1}^{N-1}c_i M_w(\widetilde{u_i}|_{V_0})$.
	It follows that 
	\begin{align}
		\sqrt{\mathcal{E}_0 (M_w (u|_{V_0}))} &=\sqrt{ \mathcal{E}_0\left( \sum_{i=1}^{N-1} c_i M_w (\widetilde{u_i}|_{V_0}) \right)} \nonumber\\
		&\leq \sum_{i=1}^{N-1} |c_i| \sqrt{\mathcal{E}_0 (M_w (\widetilde{u_i}|_{V_0}))} \nonumber\\&= \sum_{i=1}^{N-1} |c_i| \sqrt{\mathcal{E}_0 (M_w ({u_i}|_{V_0}))} \nonumber\\
		&\leq \sum_{i=1}^{N-1} |c_i| (C_2 r_w \mathcal{E}_0({u_i}|_{V_0})). \label{eq3.4.4}
	\end{align}
	
	Define new norm on ${\widetilde{\mathcal{H}}}$ by $\|h\|_{\mathcal{L}} = \sum_{i=1}^{N-1}|c_i|\sqrt{\mathcal{E}_0(\widetilde{u_i}|_{V_0})}$ for $h = \sum_{i=1}^{N-1}c_i\widetilde{u_i}|_{V_0}$. Since ${\widetilde{\mathcal{H}}}$ is finite-dimensional, $\|\cdot\|_{\mathcal{L}}$ is equivalent to $\|\cdot\|_{\widetilde{\mathcal{H}}}$. Thus, there exist constants $C_3:=C_3(V_0,{H})>0$ and $C_4:=C_4(V_0,H)>0$ such that for any $h\in \widetilde{\mathcal{H}}$, we have $C_3\|h\|_{\mathcal{L}}\leq \|h\|_{\widetilde{\mathcal{H}}}\leq C_4\|h\|_{\mathcal{L}}$.
	
	Let $A = \max_{1\leq i\leq N-1} \sqrt{\mathcal{E}_0(u_i|_{V_0})}$, we also have
	\begin{align}
		\sum_{i=1}^{N-1} |c_i| \mathcal{E}_0(u_i|_{V_0}) 
		&= \sum_{i=1}^{N-1} |c_i| \sqrt{\mathcal{E}_0(u_i|_{V_0})} \cdot \sqrt{\mathcal{E}_0(u_i|_{V_0})} \nonumber\\
		&\leq A \sum_{i=1}^{N-1} |c_i| \sqrt{\mathcal{E}_0(u_i|_{V_0})} = A \|\widetilde{u}|_{V_0}\|_{\mathcal{L}} \nonumber\\
		&\leq \frac{A}{C_3} \|\widetilde{u}|_{V_0}\|_{\widetilde{\mathcal{H}}} = \frac{A}{C_3} \sqrt{\mathcal{E}_0(\widetilde{u}|_{V_0})} = \frac{A}{C_3} \sqrt{\mathcal{E}_0(u|_{V_0})}.\label{eq3.4.5}
	\end{align}
	
Combining \eqref{eq3.4.4} and \eqref{eq3.4.5}, it follows that
	\begin{align*}
		\sqrt{\mathcal{E}_0 (M_w (u|_{V_0}))} 
		&\leq C_2  \frac{A}{C_3} r_w  \sqrt{\mathcal{E}_0(u|_{V_0})}, 
	\end{align*}
which implies the conclusion where $C = \left(\frac{C_2 A}{C_3}\right)^2$ depends only on $V_0$ and $H$.
\end{proof}

By Lemma \ref{lemma:current-equality} and Proposition
\ref{lem:osc}, we will obtain the oscillation inequality (\ref{eq_OSC}).

\begin{lemma}\rm
	\label{th:osc-bound}
	 There exists a constant $C:=C(V_0,H)$ such that for any harmonic function $u$ on $K$, for any $w\in W_m$, and for any $x,y\in F_wK$, we have
	\begin{align}
		\left| u(x)-u(y) \right| \leq C r_w \max_{1 \leq k < l \leq n} \left| u(p_k)-u(p_l) \right| \,.\label{eq_OSC}\tag{OSC}
	\end{align}
\end{lemma}

\begin{proof}
	By the maximum principle, it suffices to prove the inequality for $x,y\in V_w$. On $\widetilde{\mathcal{H}}$, define  $\left \Vert h\right \Vert _{\mathcal{M}}=\max_{1\leq m<j\leq n}|h_m-h_j|$. By the equivalence of norms in finite-dimensional spaces, there exist constants $C_5:=C_5(V_0,H)>0$ and $C_6:=C_6(V_0,H)>0$ such that for any $h\in \widetilde{\mathcal{H}}$, we have $C_5\left\Vert h\right \Vert_{\widetilde{\mathcal{H}}}\leq \left\Vert h\right \Vert_{\mathcal{M}}\leq C_6\left\Vert h\right \Vert_{\widetilde{\mathcal{H}}}$.
	
By Lemma~\ref{lem:osc}, for any $p_k,p_l\in V_0$, $k,l\in \{1,2,...,N\}$, we have
	\begin{align*}
			\left| u(F_wp_k)-u(F_wp_l)\right|&\leq \left\Vert \widetilde{M_w(u|_{V_0})}\right\Vert_{\mathcal{M}},\\&\leq C_6\sqrt{\mathcal{E}_0(\widetilde{M_w(u|_{V_0})})},\\&=C_6\sqrt{\mathcal{E}_0({M_w(u|_{V_0})})},\\&\leq C_7r_w\sqrt{\mathcal{E}_0(u|_{V_0})},\\&\leq C_7C_5r_w\max_{1 \leq k < l \leq n} \left| u(p_k)-u(p_l) \right|,	
	\end{align*}
which complete the proof. 
\end{proof}

We extend condition \eqref{eq_OSC}, which was studied for the Sierpi\'nski gasket and the Vicsek set in \cite{DRY23}, to the general setting of p.c.f. self-similar sets. In this setting, we find that \eqref{eq_OSC} takes an explicit form involving the scaling factors $r_w$, contrasting with the expression involving parameters $\delta_i$ used in \cite{THP14}.

In fact, a sharper version of \eqref{eq_OSC} for  Dirichlet form has been established in \cite[Proposition 8.12]{GY26} using Kigami's theory of resistance forms \cite{Kig12} (see also \cite[Theorem 6.27]{KS24}).

\section{Proof of Theorem \ref{thm_GRH} and Theorem \ref{corol_GRH}}
\label{holder}
In this section, we present the proof of \eqref{eq_GRH}, which is inspired  by \cite[Proposition 4.3]{DRY23} and then we give a new proof of \eqref{eq_HR}.

Denote the harmonic extension matrix corresponding to $p_i$ by
\begin{equation*}
	A_i = \begin{pmatrix}
		a_{11}^{(i)} & \cdots & a_{1N}^{(i)} \\
	    a_{21}^{(i)} & \cdots & a_{2N}^{(i)} \\
		\vdots & \vdots & \vdots \\
		a_{N1}^{(i)} & \cdots & a_{NN}^{(i)}
	\end{pmatrix}, \quad i = 1, 2, \dots, N.
\end{equation*}
These matrices satisfy $\sum_{t=1}^{N} a_{st}^{(i)} = 1$ for $s = 1, 2, \dots, N$, $a_{st}^{(i)} \geq 0$ for $1 \leq s, t \leq N$, and $a_{ii}^{(i)} = 1$ for $i = 1, 2, \dots, N$.

We will give a key property 
of $A_i$ as follows.

\begin{proposition}\rm
	\label{lem:minmatrix}
	For any positive integer $k\geq1$, let 
	\begin{equation*}
		A_i^k = A_i \times A_i \times \cdots \times A_i = \begin{pmatrix}
			a_{11}^{(i),k} & \cdots & a_{1N}^{(i),k} \\
			a_{21}^{(i),k} & \cdots & a_{2N}^{(i),k} \\
			\vdots & \vdots & \vdots \\
			a_{N1}^{(i),k} & \cdots & a_{NN}^{(i),k}
		\end{pmatrix}, \quad i = 1, 2, \dots, N, 
	\end{equation*}
	Then for any $\epsilon > 0$, there exists a positive integer $T_0 := T_0(\epsilon)$ such that   
	\begin{align}
		\min_{1 \leq i \leq N} \min_{1 \leq j \leq N} a_{ji}^{(i),T_0} \geq \frac{1}{2} + \epsilon.
	\end{align}
\end{proposition}

\begin{proof}
Without loss of generality, we only prove the statement for $A_1$.
Suppose $u$ is harmonic on $K$. For any vector $(u(p_1), u(p_2), \dots, u(p_N))^T$, by Theorem~\ref{th:osc-bound}, we have
	\begin{equation*}
		\begin{split}
			\max_{1 \leq i < j \leq N} \left| \sum_{t=1}^{N} a_{jt}^{(1),k} u(p_t) - \sum_{t=1}^{N} a_{it}^{(1),k} u(p_t) \right| & \leq C \rho^{(\beta-\alpha)k} \max_{1 \leq i < j \leq N} |u(p_i) - u(p_j)| \\
			& \to 0 \quad (k \to \infty).
		\end{split}
	\end{equation*}
	In particular, for every $2 \leq j \leq N$, we have
	\begin{equation*}
		\left| \sum_{t=1}^{N} a_{jt}^{(1),k} u(p_t) - u(p_1) \right| \to 0 \quad (k \to \infty).
	\end{equation*}
	
Now, choose a harmonic function $u$ on $K$ such that $(u(p_1), u(p_2), \dots, u(p_N))^T = (1, 0, \dots, 0)^T$. Then we obtain
	\begin{equation*}
		|a_{j1}^{(1),k} - 1| \to 0 \quad (k \to \infty),
	\end{equation*}
	for every $j = 2, 3, \dots, N$.
Hence, there exists $T_1$ such that 
$$\min_{1 \leq j \leq N} a_{j1}^{(1),T_1} \geq \frac{1}{2} + \epsilon.$$
	
Let $T_0 = \max_{1 \leq i \leq N} \{T_i\}$. By induction,
we have 
 $$\min_{1 \leq i \leq N} \min_{1 \leq j \leq N} a_{ji}^{(i),T_0} \geq \frac{1}{2} + \epsilon,$$
hence we complete the proof.
\end{proof}

\begin{proof}[Proof of Theorem \ref{thm_GRH}]

	By Proposition \ref{lem:minmatrix}, we can pick up $T_0$ such that $$\min_{1\leq i\leq N}\min_{1\leq j\leq N}a_{ji}^{(i),T_0}\geq \frac{5}{6}.$$
	
	For any ball $B:=B(x_0,r)$, when $r\in (0,\rho^{-T_0})$, the \eqref{eq_GRH} inequality obviously holds. 
	We suppose that $u$ is harmonic in $lB$($l$ is the diameter of the 1-skeleton of $X$).
	We now assume $r\geq\rho ^{-T_0}$, and without loss of generality, let $\rho^{-k}\leq r < \rho^{-(k+1)}$ for some $k\geq T_0$. For any $x\in B\setminus G$, there exist $p,q\in G\cap 2B$ with $|p-q|=1$ such that $x\in (p,q)\subseteq 2B$, hence $|\nabla u(x)|=|u(p)-u(q)|$. Since there exists a $k$-skeleton $W$ satisfying $p,q\in W\subseteq lB$, and we have
	\begin{equation*}
		m(W)\leq m(lB)\leq C\left(\frac{l}{\rho^{k+1}}\right)^{\alpha} = C_1 M^{k+1},
	\end{equation*}
	and there exists $C_2>1$ such that $C_2^{-1}M^k\leq m(W)\leq C_2M^k$.
	
	
	Let $q_1,q_2,\cdots,q_N$ be the boundary points of $W$. Let $F:\mathbb{R}^2\to\mathbb{R}^2$ be the affine mapping that maps $p_i$ to $q_i$, $i=1,2,3,\cdots,N$. Let $v$ be the harmonic function on the p.c.f. self-similar set $K$ with $v(p_i)=u(q_i)$, $i=1,2,3,\cdots,N$. Noting that $W\cap V=F(V_k)$, we have $v=u\circ F$ on $V_k$ or $u=v\circ F^{-1}$ on $W\cap V$. Let $i_1,\ldots,i_k\in\{1,2,3,\cdots,M\}$ satisfy $F^{-1}(p),F^{-1}(q)\in f_{i_1}\circ\ldots\circ f_{i_k}(K)$. By Theorem \ref{th:osc-bound}, we have 
   \begin{align*}
	|u(p)-u(q)|&=|v(F^{-1}(p))-v(F^{-1}(q))|\\
	&\leq\mathrm{Osc}(v,f_{i_1}\circ\ldots\circ f_{i_k}(K))\\
	&=Cr_{w}\max\{|v(p_i)-v(p_j)|:i,j=1,2,3,\cdots,N\}\\
	&=Cr_{w}\max\{|u(q_i)-u(q_j)|:i,j=1,2,3,\cdots,N\}.
   \end{align*}
   
   
	 Without loss of generality, assume $u(q_1)=\max_{1\leq i\leq N}|u(q_i)| > 0$. Take an $(k-T_0)$-skeleton $W_0$ that has $q_1$ as one of its boundary points and satisfies $W_0\subseteq W$. Denote the boundary points of $W_0$ by $t_1,t_2,\dots,t_n$. We obtain
	\begin{equation*}
		\begin{split}
			u(t_j) & = \sum_{i=1}^{N}a_{ji}^{(1),T_0}u(q_i) \\
			& = a_{j1}^{(1),T_0}u(q_1) + \sum_{i=2}^{N}a_{ji}^{(1),T_0}u(q_i) \\
			& \geq \left(a_{j1}^{(1),T_0} - \sum_{i=2}^{N}a_{ji}^{(1),T_0}\right)u(q_1) \\
			& = (2a_{j1}^{(1),T_0} - 1)u(q_1) \\
			& \geq \frac{2}{3}u(q_1),j=1,2...N,
		\end{split}
	\end{equation*}
	where the fourth equality uses the fact that $\sum_{i=1}^{N}a_{ji}^{(1),T_0}=1$.
	
	By the maximum principle, we have
	\begin{equation*}
		u(t) \geq \frac{2}{3}u(q_1) > 0,
	\end{equation*}
	for any $t$ on $W_0$.
	
Therefore,
	\begin{align*}
			|\nabla u(x)|  &= |u(p)-u(q)|  \leq C\rho^{(\beta-\alpha)k} \max_{1\leq i < j\leq n}|u(q_i)-u(q_j)| \nonumber\\
			& \leq 2C\rho^{(\beta-\alpha)k}u(q_1)  \leq 2C\rho^{(\beta-\alpha)k}\frac{3}{2} \dashint_{W_0} |u| \, dm 
			\leq C_3 \rho^{(\beta-\alpha)k}\dashint_{lB}|u|dm\\&\leq C_4\frac{1}{r^{\beta-\alpha}}\dashint_{lB}|u|dm,
	\end{align*}
	If $u$ is a harmonic function in $2B$, covering $B$ by at most countable family of balls like $B(y_i,\frac{r}{2l}),y_i\in B$, we conclude that
	\begin{align}
			\left\Vert \left\vert \nabla u\right\vert\right\Vert_{L^{\infty}(B)}&\leq\sup_{i}\left\Vert \left\vert \nabla u\right\vert\right\Vert_{L^{\infty}(B(y_i,\frac{r}{2l}))}\nonumber\\&\leq C_4\sup_{i}(2l)^{\beta-\alpha}\frac{1}{r^{\beta-\alpha}}\dashint_{B(y_i,\frac{r}{2})}|u|dm\nonumber\\
			&\leq  \frac{C_5}{r^{(\beta-\alpha)}}\dashint_{2B}|u|dm,\label{eq4.3}
	\end{align}
which complete the proof.
\end{proof}

Next we will prove the H\"older regularity \eqref{eq_HR} of harmonic function in 
bounded and unbounded p.c.f. self-similar set
$K_n$.

According to \cite[Proposition 2.5]{GL20A}, a p.c.f. self-similar set $K$ satisfies Condition (H). This means there exists a constant $c_H = c_H(K) > 0$ with the following property: for any two words $w, w'$ of length $m \geq 1$, points $x$ and $y$ satisfy $d(x,y) < c_H\rho^m$ if and only if they lie in the same or adjacent $m$-cells. We observe that Condition (H) also holds in the unbounded case (see \cite{GYZ26}). 

\begin{proof}[Proof of Theorem \ref{corol_GRH}]

	Let $B=B(x_0,r)$ and let $u$ be a harmonic function in $l_0B$, where $l_0$ is left to be determined later. 
	We may assume there exists $k\in \mathbb{Z}$ such that $\rho^{k+1}\leq r<\rho^{k}$. 
	For any $x,y\in B$, we may assume there exists $m\in\mathbb{Z}$ such that 
	$c_H\rho^{m+1}\leq d(x,y)< c_H\rho^m$. Since $x,y\in B$, we have 
	$c_H\rho^{m+1}\leq d(x,y)\leq 2r<2\rho^k$.
	We find that a positive integer $n_0$ such that 
	for $n_0+\frac{\log(2/c)}{\log\rho}>0$, we have   $m-(k-n_0-1)>1$.
	
	From the construction of $K_\infty$, we know that for sufficiently large $n$, $x,y\in K_n$. 
	By condition $(H)$, there exist $m+n$ cells $W^{1}$ and $W^{2}$ of $K_n$ such that 
	$x\in W^{1}$, $y\in W^{2}$, and $W^1\cap W^2\neq\emptyset$.  
	We can extend $W^1$, $W^2$ to $Y_{k-n_0-1}^1$, $Y_{k-n_0-1}^2$ respectively, where  $Y_{k-n_0-1}^1$, $Y_{k-n_0-1}^2$ are translates of $\rho^{k-n_0-1}K$,
	then $W^1$, $W^2$ are the $m-(k-n_0-1)$ cells of $Y_{k-n_0-1}^1$, $Y_{k-n_0-1}^2$ respectively. 
	We can pick up $l_0:=\left(\rho^{-(n_0+2)}+1\right)$ such that $Y_{k-n_0-1}^1,Y_{k-n_0-1}^2\subseteq l_0B$.
	
	Let $z\in W^1\cap W^2$. For $x$, there exists a chain $\{x_{i}\}_{i\geq0}\subseteq V_{\infty}$ 
	between $z$ and $x$, where $x_0=z$, $d(x_{i+1},x_i) \simeq \rho^{m+i}$, $\lim_{i\rightarrow\infty} x_i=x$. 
	Therefore, by Lemma \ref{th:osc-bound}, we have 
	\begin{align*}
		|u(x_{i+1})-u(x_{i})|
		&\leq C_1 \rho^{(\beta-\alpha)(m-(k-n_0-1)+i)} 
		\max_{\substack{p_s,p_j \text{ are vertices of } Y_{k-n_0-1}^1,\\ 
				1\leq s<j\leq N}} |u(p_s)-u(p_j)| \\
		&=C_2\rho^{(\beta-\alpha)(m-k+i)} 
		\max_{\substack{p_s,p_j \text{ are vertices of } Y_{k-n_0-1}^1,\\ 
				1\leq s<j\leq N}} |u(p_s)-u(p_j)|.
	\end{align*}
	
Note that $m(Y_{k-n_0-1}^1)$ is comparable to $m(B)$, 
we have 
	\begin{equation*}
		\max_{\substack{p_s,p_j \text{ are vertices of } Y_{k-n_0-1}^1,\\ 
				1\leq s<j\leq N}} |u(p_s)-u(p_j)|
		\leq C_3 \dashint_{l_0B}|u|\,dm.
	\end{equation*}
Hence,
	\begin{align*}
		\left|u(z)-u(x)\right|
		&\leq C_2C_3 \sum_{i=0}^{\infty} 
		\rho^{(\beta-\alpha)(m-k+i)}\dashint_{l_0B}|u|\,dm  \\
		&\leq \frac{C_2C_3}{1-\rho^{\beta-\alpha}} 
		\rho^{(\beta-\alpha)(m-k)} \dashint_{l_0B}|u|\,dm.
	\end{align*}

	Similarly, we obtain
	\[
	\left|u(z)-u(y)\right|
	\leq \frac{C_2C_3}{1-\rho^{\beta-\alpha}} 
	\rho^{(\beta-\alpha)(m-k)} \dashint_{l_0B}|u|\,dm.
	\]
	Therefore,
	\begin{equation*}
		\begin{split}
			\left|u(x)-u(y)\right|
			&\leq \left|u(x)-u(z)\right| + \left|u(z)-u(y)\right| \\
			&\leq \frac{2C_2C_3}{1-\rho^{\beta-\alpha}} 
			\rho^{(\beta-\alpha)(m-k)} \dashint_{l_0B}|u|\,dm \\
			&\leq C_4\left(\frac{d(x,y)}{r}\right)^{\beta-\alpha} 
			\dashint_{l_0B}|u|\,dm.
		\end{split}
	\end{equation*}
	
	Now assume $u$ is a harmonic function in $2B$. For any $x,y\in B$, when 
	$d(x,y)\leq \frac{r}{4l_0}$, we observe that 
	$y\in B(x,\frac{r}{2l_0})\subseteq B(x,\frac{r}{2})\subseteq 2B$, thus
	\begin{equation*}
		\left|u(x)-u(y)\right|
		\leq C_4(2l_0)^{\beta-\alpha}\left(\frac{d(x,y)}{r}\right)^{\beta-\alpha}
		\dashint_{B(x,\frac{r}{2})}|u|\,dm
		\leq C_6\left(\frac{d(x,y)}{r}\right)^{\beta-\alpha}
		\dashint_{2B}|u|\,dm.
	\end{equation*}
	When $d(x,y)\geq \frac{r}{4l_0}$, we can prove that for any $x\in B$,
	\begin{equation*}
		|u(x)|\leq \max_{x\in B} |u(x)|\leq C_7 \dashint_{2B}|u|\,dm.
	\end{equation*}
	Therefore, we can choose a sufficiently large constant $C_8$, 
	depending only on $l_0$, such that
	\begin{equation*}
		\left|u(x)-u(y)\right|
		\leq 2C_7 \dashint_{2B}|u|\,dm
		\leq C_8\left(\frac{d(x,y)}{r}\right)^{\beta-\alpha}
		\dashint_{2B}|u|\,dm.
	\end{equation*}
	
	In summary, for any $x,y \in B$, we have 
	\begin{equation*}
		\left|u(x)-u(y)\right|
		\leq \max\{C_6,C_8\}\left(\frac{d(x,y)}{r}\right)^{\beta-\alpha}
		\dashint_{2B}|u|\,dm,
	\end{equation*}
	which completes the proof.

\end{proof}

\section{Example}
\label{sec:example}
In this section, we primarily present two classic examples of p.c.f. self-similar sets, namely the Sierpi\'nski gasket and the Vicsek set. Our goal is to prove Theorems \ref{thm_GRH} and \ref{corol_GRH} on these sets. The key lies in establishing the oscillation inequality in Lemma \ref{th:osc-bound}. Our proof differs from that Proposition 4.1 and Proposition 4.3 in \cite[Proposition 4.1 and Proposition 4.3]{DRY23}, as we provide a unified approach.
Beside the Sierpi\'nski gasket and the Vicsek set, we also give  a Eyebolted Vicsek cross, which is not a nested fractal (see \cite{GL20B}), to show the main result.

\begin{figure}[ht]
	\centering
	\begin{minipage}[t]{0.4\textwidth}
		\centering
		\includegraphics[width=\textwidth]{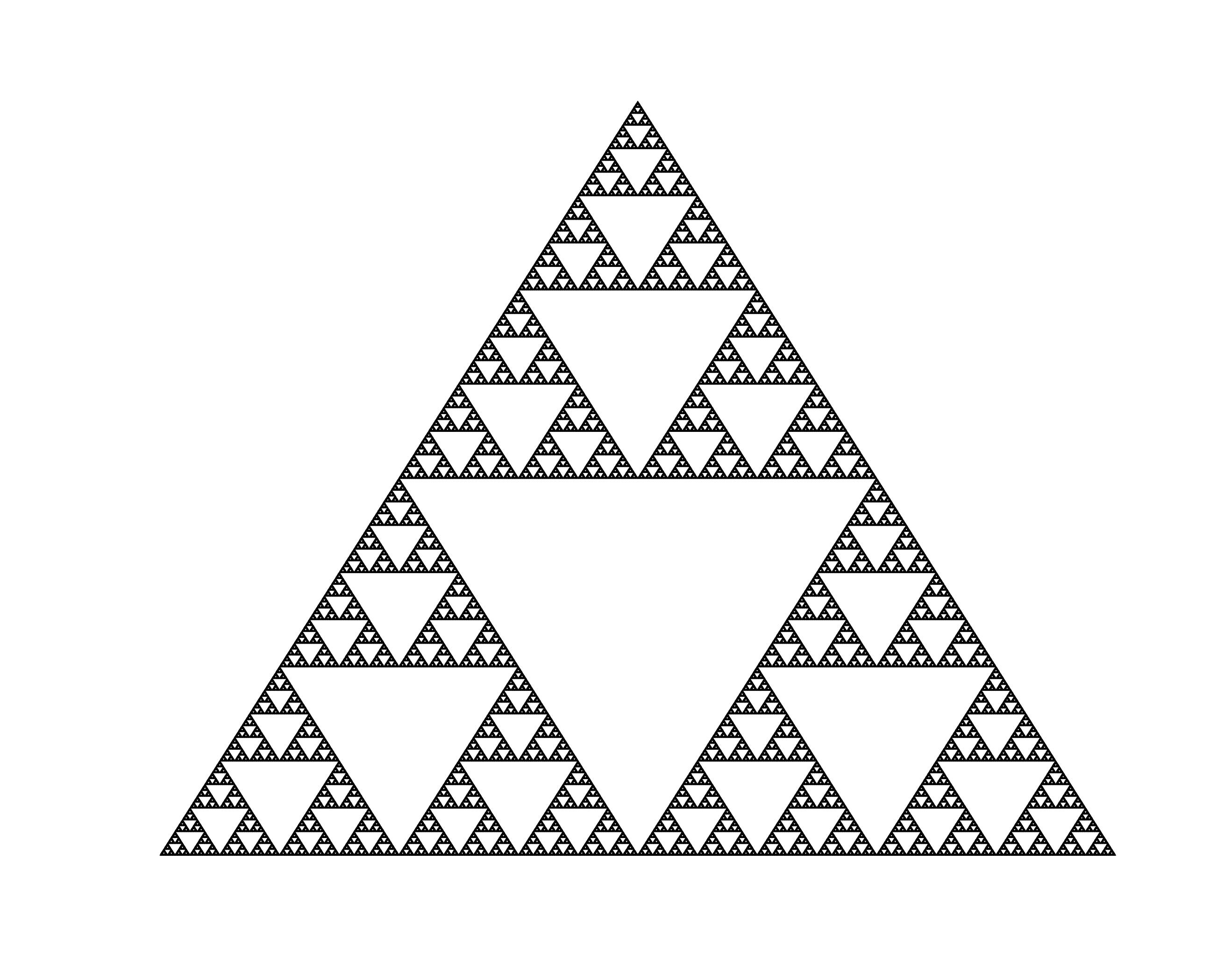}
		\caption{The Sierpi\'nski Gasket}\label{fig_SG}
	\end{minipage}
	\begin{minipage}[t]{0.4\textwidth}
		\centering
		\includegraphics[width=\textwidth]{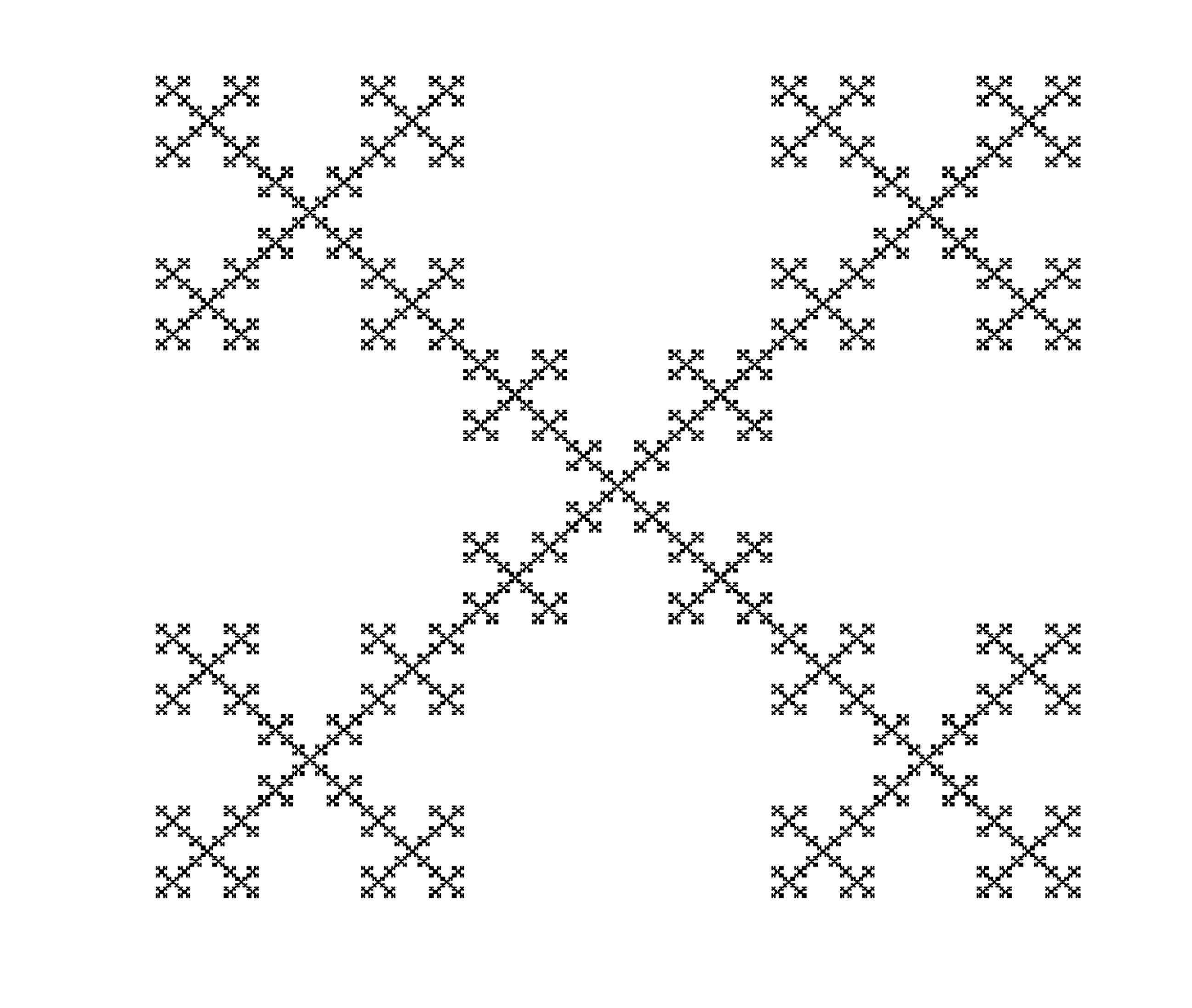}
		\caption{The Vicsek set}\label{fig_V}
	\end{minipage}
\end{figure}

\begin{example}
[The Sierpi\'nski gasket]\label{exp1}\rm
 In $\mathbb{R}^2$, let $ p_1 = (0, 0)$, $ p_2 = (1, 0) $ and $ p_3 = (\frac{1}{2}, \frac{\sqrt{3}}{2}) $.  Let $\{F_i\}_{i=1}^3$
 be the IFS of contractive similitudes on $\mathbb{R}^2$ such that
 \begin{align}
   F_i(x) = \frac{1}{2}(x + p_i), 1\leq i\leq 3,
 \end{align}
Sierpi\'nski gasket $K$ is a unique non-empty compact set in $\mathbb{R}^2$ satisfying $ K = \bigcup_{i=1}^{3} F_i(K) $ see Figure \ref{fig_SG}.
 Let $ V_0 = \{ p_1, p_2, p_3 \} $ and $ V_{n+1} = \bigcup_{i=1}^{3} F_i(V_n) $ for any $ n \geq 0 $. Then $\{ V_n \}_{n \geq 0}$ is an increasing sequence of finite subsets of $ K $ and the closure of $ \bigcup_{n \geq 0} V_n $ is $ K $. 
\end{example}
We will construct the local regular Dirichlet form ($\mathcal{E},\mathcal{F}$) as \eqref{DF_Def} on the Sierpi\'nski gasket (also see \cite{Kig89}), 
where
\begin{align*}
\mathcal{E}_{m}(u, u) = \left(\frac{5}{3}\right)^m \cdot \frac{1}{2} \sum_{w \in W_m} \sum_{x,y \in V_w} (u(x) - u(y))^2.
\end{align*}
By Lemma \ref{th:osc-bound}, we will give the  oscillation inequality of Sierpi\'nski gasket as follows. 

Firstly, for any $i,j\in\{1,2,3\}$, we assume that $u$ is harmonic function on $K$ with $u(p_i)=1$ and
$u(p_j)=0$ for all $j\neq i$,
\begin{align*}
\mathcal{E}_0(u|_{V_0}) = 2.
\end{align*}
Also see Figure \ref{fig3}.
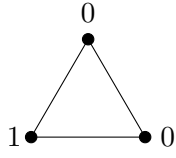
\begin{figure}[ht]
	\centering
	\begin{tikzpicture}[scale=1.5]
		\draw(0,0)--(1,0)--(1/2,1.73/2)--(0,0)--cycle;
		\node at (-0.15,0){$1$};
		\node at (1.20,0){$0$};
		\node at (0.5,1.1){$0$};
		\draw[fill=black] (0,0) circle (0.05) ;
		\draw[fill=black] (1,0) circle (0.05) ;
		\draw[fill=black] (1/2,1.73/2) circle (0.05);
	\end{tikzpicture}
	\caption{ harmonic function $u$ on $V_0$ for the Sierpi\'nski gasket}\label{fig3}
\end{figure}

By Lemma~\ref{lemma:current-equality}, for any $w\in W_m$ and for any $x,y\in V_w$,

\begin{align*}
|u(x) - u(y)| \leq \left(\frac{3}{5}\right)^m \mathcal{E}_0 (u|_{V_0})
\end{align*}

Therefore,
\begin{align*}
		&\mathcal{E}_0 (M_w (u \big|_{V_0}))		
		= \frac{1}{2} \sum_{x,y \in V_0} (u(F_w x) - u(F_w y))^2	\\
		&\leq 3 \max_{x,y \in V_0} |u(F_w x) - u(F_w y)|^2
        \leq 6\left(\frac{3}{5}\right)^{2m} \mathcal{E}_0 (u \big|_{V_0})
\end{align*}

For any harmonic function $u = \sum_{i=1}^3 c_i u_i$ on $K$, where $u_i$ is a harmonic function with $u_i(p_i)=1$ and
$u_i(p_j)=0$ for all $j\neq i$.  Without loss of generality assume $\sum_{i=1}^{3}u(p_i)=0$, i.e., $\sum_{i=1}^{3}c_i=0$.
Thus, we have 
\begin{align*}
		\mathcal{E}_0^\frac{1}{2} (M_w (u \big|_{V_0})) \leq \sum_{i=1}^3 |c_i| \mathcal{E}_0^\frac{1}{2} (M_w (u_i \big|_{V_0}))
		\leq \sqrt{6}\left(\frac{3}{5}\right)^{m}\sum_{i=1}^3|c_i| \mathcal{E}_0^{\frac{1}{2}} (u_i \big|_{V_0})
\end{align*}

Using the Cauchy-Schwarz inequality and $\sum_{i=1}^{3}c_i=0$, we can prove
\begin{align*}
		\frac{1}{2}\sum_{i=1}^{3}|c_i|\mathcal{E}_0^{\frac{1}{2}}(u_i|_{V_0})
		=\frac{\sqrt{2}}{2}\sum_{i=1}^{3}|c_i|\leq \left( \sum_{1 \leq i < j \leq 3} (c_i - c_j)^2 \right)^{1/2} =\mathcal{E}_0^{1/2} \left( \sum_{i=1}^3 c_i u_i \right).
\end{align*}
it implies that 
\begin{align*}
\mathcal{E}_0^\frac{1}{2} (M_w (u \big|_{V_0}))  \leq 2\sqrt{6}\left(\frac{3}{5}\right)^{m}  \mathcal{E}_0^{1/2}(u \big|_{V_0}).
\end{align*}

Thus, for any $w\in W_m$, $x,y\in V_w$, we have 
\begin{align}
		|u(x)-u(y)| \leq \mathcal{E}_0^{1/2} (M_w u \big|_{V_0})
		\leq 6\sqrt{2}\left(\frac{3}{5}\right)^m \max_{1 \leq k,l \leq 3}|u(p_k)-u(p_l)|.\label{sg}
\end{align}

Applying \eqref{sg} in Theorem \ref{thm_GRH}, we obtain the \eqref{eq_GRH}  holds in cable system induced by Sierpi\'nski gasket, which is also considered in  \cite[Proposition 4.3]{DRY23}. Furthermore, we also prove that \eqref{eq_HR} holds by \eqref{sg} in the bounded and unbounded Sierpi\'nski gasket.

\begin{example}[The Vicsek set]\label{exp3}\rm
	Let $p_1 = (0, 0)$,$p_2 = (1, 0)$, $p_3 = (1, 1)$, $p_4 = (0, 1)$, $p_5 = \left( \frac{1}{2}, \frac{1}{2} \right)$
	be the four corners and center of the unit square in the plane. Define
	\begin{align}
		F_i(x) = \frac{1}{3} (x - p_i) + p_i \quad (1 \leq i \leq 5).
	\end{align}
	The \emph{Vicsek set} $K$ is determined by $ K = \bigcup_{i=1}^{5} F_i(K)$ see Figure \ref{fig_V}. It is  a nested fractal, and the boundary is $ V_0 = \{ p_1, p_2, p_3, p_4 \} $.
	 Let  $ V_{n+1} = \bigcup_{i=1}^{4} F_i(V_n) $ for any $ n \geq 0 $. Then $\{ V_n \}_{n \geq 0}$ is an increasing sequence of finite subsets of $ K $ and the closure of $ \bigcup_{n \geq 0} V_n $ is $ K $. 
\end{example}	
We will construct the local regular Dirichlet form ($\mathcal{E},\mathcal{F}$) as \eqref{DF_Def} on the Vicsek set, 
where
$$
\mathcal{E}_{m}(u, u) = 3^m \cdot \frac{1}{2} \sum_{w \in W_m} \sum_{x,y \in V_w} (u(x) - u(y))^2, \quad u \text{ is a function on } V_m.
$$
By Lemma \ref{th:osc-bound}, we will give the  oscillation inequality of Vicsek set as follows.

Firstly, for any $i,j\in\{1,2,3,4\}$, we assume that $u$ is harmonic function on $K$ with $u(p_i)=1$ and
$u(p_j)=0$ for all $j\neq i$,
$$
\mathcal{E}_0(u|_{V_0}) = 3.
$$
Also see Figure \ref{fig4}.
\begin{figure}[ht]
	\centering
	\begin{tikzpicture}[scale=1.5]
		\draw(0,0)--(1,0)--(1,1)--(0,1)--(0,0)--cycle;
		\draw(0,0)--(1,1)--(0,0)--cycle;
		\draw(1,0)--(0,1)--(1,0)--cycle;
		\node at (-0.15,0){$1$};
		\node at (1.20,0){$0$};
		\node at (1.2,1){$0$};
		\node at (-0.15,1){$0$};
		\draw[fill=black] (0,0) circle (0.05) ;
		\draw[fill=black] (1,0) circle (0.05) ;
		\draw[fill=black] (1,1) circle (0.05);
		\draw[fill=black] (0,1) circle (0.05);
	\end{tikzpicture}
	\caption{harmonic function $u$ on $V_0$ for the Vicsek set}\label{fig4}
\end{figure}
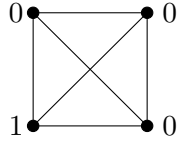

By Lemma~\ref{lemma:current-equality}, for any $w\in W_m$ and for any $x,y\in V_w$,
$$
|u(x) - u(y)| \leq \frac{1}{3^m} \mathcal{E}_0 (u|_{V_0})
$$

Therefore,
\begin{align*}
	\mathcal{E}_0 (M_w (u\big|_{V_0}))&		
	= \frac{1}{2} \sum_{x,y \in V_0} (u(F_w x) - u(F_w y))^2	
	,\\&\leq 6 \max_{x,y \in V_0} |u(F_w x) - u(F_w y)|^2
	,\\&\leq \frac{6}{3^{2m}} \mathcal{E}_0 ^2(u \big|_{V_0}) = \frac{18}{3^{2m}} \mathcal{E}_0 (u \big|_{V_0}).
\end{align*}

For any harmonic function $u = \sum_{i=1}^4 c_i u_i$ on $K$, where $u_i$ is harmonic function on $K$ with $u_i(p_i)=1$ and
$u_i(p_j)=0$ for all $j\neq i$. Without loss of generality assume $\sum_{i=1}^{4}u(p_i)=0$, i.e., $\sum_{i=1}^{4}c_i=0$.
Thus we have 
\begin{align*}
	\mathcal{E}_0^\frac{1}{2} (M_w (u \big|_{V_0})) \leq \sum_{i=1}^4 |c_i| \mathcal{E}_0^\frac{1}{2} (M_w (u_i \big|_{V_0}))\leq \frac{\sqrt{18}}{3^m} \sum_{i=1}^4 |c_i| \mathcal{E}_0^{\frac{1}{2}} (u_i \big|_{V_0}).
\end{align*}
Using the Cauchy-Schwarz inequality and $\sum_{i=1}^{4}c_i=0$, we can prove
\begin{align*}
	\frac{1}{2}\sum_{i=1}^{4}|c_i|\mathcal{E}_0^{\frac{1}{2}}(u_i|_{V_0})
	=\frac{\sqrt{3}}{2}\sum_{i=1}^{4}|c_i|\leq \left( \sum_{1 \leq i < j \leq 4} (c_i - c_j)^2 \right)^{1/2} =\mathcal{E}_0^{1/2} \left( \sum_{i=1}^4 c_i u_i \right).
\end{align*}
It implies
\[
\mathcal{E}_0^\frac{1}{2} (M_w (u \big|_{V_0}))  \leq \frac{6\sqrt{2}}{3^m}  \mathcal{E}_0^{1/2}(u \big|_{V_0}).
\].

Thus for any $w\in W_m$, $x,y\in V_w$, we have
\begin{align}	
	|u(x)-u(y)| \leq \mathcal{E}_0^{1/2} (M_w u \big|_{V_0})\leq \frac{2 \sqrt{18} }{3^m}  \mathcal{E}_0^{1/2} (u \big|_{V_0})
	\leq \frac{12\sqrt{3}}{3^m} \max_{1 \leq k,l \leq 4}|u(p_k)-u(p_l)|,\label{vc}
\end{align}
which implies the oscillation inequality of harmonic function. 

Applying \eqref{vc} in Theorem \ref{thm_GRH}, we obtain the \eqref{eq_GRH}  holds in cable system induced by Vicsek set, which is also considered in  \cite[Proposition 4.1]{DRY23}. Furthermore, we also prove that \eqref{eq_HR} holds by \eqref{vc} in the bounded and unbounded Vicsek set.

Finally, we will give a example besides Sierpi\'nski gasket and Vicsek set, for example, eyebolted Vicsek cross introduced by Gu and Lau in \cite[Section 7]{GL20B}. Further example also be found in \cite{GL20A,GL20B,Kig01book} and the references therein. 

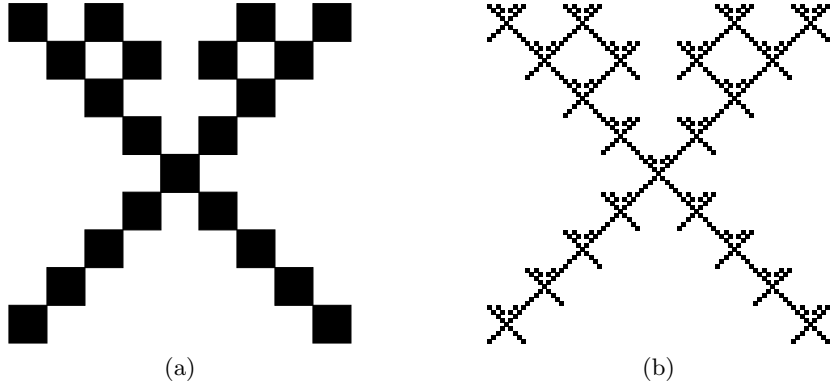
\begin{figure}[htbp]
	\centering
	\subfloat[]{
		\begin{tikzpicture}[scale=0.5]
			\draw[fill=black] (0,0)--(1,0)--(1,1)--(0,1)--cycle;
			\draw[fill=black] (1,1)--(2,1)--(2,2)--(1,2)--cycle;
			\draw[fill=black] (2,2)--(3,2)--(3,3)--(2,3)--cycle;
			\draw[fill=black] (3,3)--(4,3)--(4,4)--(3,4)--cycle;
			\draw[fill=black] (4,4)--(5,4)--(5,5)--(4,5)--cycle;
			\draw[fill=black] (5,5)--(6,5)--(6,6)--(5,6)--cycle;
			\draw[fill=black] (6,6)--(7,6)--(7,7)--(6,7)--cycle;
			\draw[fill=black] (7,7)--(8,7)--(8,8)--(7,8)--cycle;
			\draw[fill=black] (8,8)--(9,8)--(9,9)--(8,9)--cycle;
			\draw[fill=black] (8,0)--(9,0)--(9,1)--(8,1) --cycle;
			\draw[fill=black] (7,1)--(8,1)--(8,2)--(7,2)--cycle;
			\draw[fill=black] (6,2)--(7,2)--(7,3)--(6,3)--cycle;
			\draw[fill=black] (5,3)--(6,3)--(6,4)--(5,4)--cycle;
			\draw[fill=black] (4,4)--(5,4)--(5,5)--(4,5)--cycle;
			\draw[fill=black] (3,5)--(4,5)--(4,6)--(3,6)--cycle;
			\draw[fill=black] (2,6)--(3,6)--(3,7)--(2,7)--cycle;
			\draw[fill=black] (1,7)--(2,7)--(2,8)--(1,8)--cycle;
			\draw[fill=black] (0,8)--(1,8)--(1,9)--(0,9)--cycle;
			\draw[fill=black] (2,8)--(3,8)--(3,9)--(2,9)--cycle;
			\draw[fill=black] (3,7)--(4,7)--(4,8)--(3,8)--cycle;
			\draw[fill=black] (5,7)--(6,7)--(6,8)--(5,8)--cycle;
			\draw[fill=black] (6,8)--(7,8)--(7,9)--(6,9)--cycle;		
		\end{tikzpicture}
	}\hspace{40pt}
	\subfloat[]{
		\begin{tikzpicture}[scale=0.5]
			\def\levelOneBlocks{
				{0,0}, {1,1}, {2,2}, {3,3}, {4,4}, {5,5}, {6,6}, {7,7}, {8,8},
				{8,0}, {7,1}, {6,2}, {5,3}, {3,5}, {2,6}, {1,7}, {0,8},
				{2,8}, {3,7}, {5,7}, {6,8}
			}
			
			\def\originalBlocks{
				{0,0}, {1,1}, {2,2}, {3,3}, {4,4}, {5,5}, {6,6}, {7,7}, {8,8},
				{8,0}, {7,1}, {6,2}, {5,3}, {3,5}, {2,6}, {1,7}, {0,8},
				{2,8}, {3,7}, {5,7}, {6,8}
			}
			
			\foreach \i/\j in {0/0,1/1,2/2,3/3,4/4,5/5,6/6,7/7,8/8,
				8/0,7/1,6/2,5/3,3/5,2/6,1/7,0/8,
				2/8,3/7,5/7,6/8} {
				\foreach \I/\J in {0/0,1/1,2/2,3/3,4/4,5/5,6/6,7/7,8/8,
					8/0,7/1,6/2,5/3,3/5,2/6,1/7,0/8,
					2/8,3/7,5/7,6/8} {
					\pgfmathsetmacro{\x}{\i + \I/9}
					\pgfmathsetmacro{\y}{\j + \J/9}
					\fill[black] (\x, \y) rectangle ++(1/9, 1/9);
				}
			}
			
		\end{tikzpicture}
	}
	\caption{A eyebolted Vicsek cross}
	\label{fig5}
\end{figure}

\begin{example}[The eyebolted Vicsek cross]\label{exp3}\rm
	In ${\Bbb R}^2$, In $\mathbb{R}^2$, let $p_1 = (0, 0)$,$p_2 = (1, 0)$, $p_3 = (1, 1)$, $p_4 = (0, 1)$, $p_5 = \left( \frac{1}{2}, \frac{1}{2} \right)$
	be the four corners and center of the unit square in the plane.
	Divide the unit square into a mesh of sub-squares of size $1/9$, and pick $21$ sub-squares as shown in Figure \ref{fig5}.
	Let $\{a_i\}_{i=1}^{21}$ be the center of these sub-squares. Let $\{F_i\}_{i=1}^{21}$ be the IFS on $\mathbb{R}^2$ with
	\begin{align}
		F_i(x) = \frac{1}{9}(x-a_i) + a_i,  \qquad 1\leq i \leq 21.
	\end{align}
	The \emph{eyebolted Vicsek cross} $K=\bigcup_{i=1}^{21} F_i(K)$ is a p.c.f. self-similar set with boundary $V_0=\{p_1,p_2,p_3,p_4\}$, which is  not a nested fractal.
\end{example}	

The Hausdorff dimension of $K$ is $\alpha=\log21/\log9$, and the self-similar measure with the natural weight is the normalized $\alpha-$dimensional Hausdorff measure $\mu$ on $K$.
The walk dimension of $K$ is $\beta=\frac{\log 21+\log(35/4)}{\log 9}$. By \cite[Theorem 6.1]{GL20B}, there exist  a local regular Dirichlet
form on $K$ with the renormalization factor $r=\frac{4}{35}$.

By Lemma \ref{th:osc-bound}, thus for any $w\in W_m$, $x,y\in V_w$, there exist $C>0$ such that
\begin{align*}
	\left| u(x)-u(y) \right| \leq C \left(\frac{4}{35}\right)^{m} \max_{1 \leq k < l \leq 4} \left| u(p_k)-u(p_l) \right|.
\end{align*} 
Then, we also prove that \eqref{eq_GRH} holds in the cable system induced by eyebolted Vicsek cross, and \eqref{eq_HR} holds in the bounded and unbounded eyebolted Vicsek cross.

\bibliographystyle{plain}

\begin{thebibliography}{10}
	
\bibitem{ACDH04}
P.~Auscher, T.~Coulhon, X.~T. Duong, and S.~Hofmann.
\newblock Riesz transform on manifolds and heat kernel regularity.
\newblock {\em Ann. Sci. \'Ecole Norm. Sup.}, 37(6):911--957, 2004.



\bibitem{Bar17}
M.~T. {Barlow} 
\newblock {\em Random walks and heat kernels on graphs}, volume 438 of {\em London Mathematical Society Lecture Note Series},
\newblock Cambridge University Press, Cambridge, 2017.


\bibitem{BB04}
M.~T. {Barlow} and R.~F. {Bass}.
\newblock Stability of parabolic Harnack inequalities.
\newblock {\em Trans. Amer. Math. Soc.}, 356(4):1501--1533, 2004.


\bibitem{BC23}
F.~{Baudoin} and L.~{Chen}.
\newblock {Heat kernel gradient estimates for the Vicsek set}.
\newblock {\em Math. Nachr.},  297 (12):4450--4477, 2024.

\bibitem{CCH06}
G.~{Carron}, T.~{Coulhon} and A.~{Hassell}. 
\newblock {Riesz transform and $L_p$-cohomology for manifolds with Euclidean ends}.
\newblock {\em {Duke Math. J.}}, 133(1):59–93, 2006.

\bibitem{Chen15}
L.~{Chen}.
\newblock {Sub-Gaussian heat kernel estimates and quasi Riesz transforms for
	\(1\leq p\leq 2\)}.
\newblock {\em {Publ. Mat., Barc.}}, 59(2):313--338, 2015.

\bibitem{CCFR17}
Li~{Chen}, T.~{Coulhon}, J.~{Feneuil}, and E.~{Russ}.
\newblock Riesz transform for {$1\le p\le 2$} without {G}aussian heat kernel
bound.
\newblock {\em J. Geom. Anal.}, 27(2):1489--1514, 2017.

\bibitem{CD99}
T.~{Coulhon} and X.~T. {Duong}.
\newblock {Riesz transforms for \(1\leq p\leq 2\)}.
\newblock {\em {Trans. Am. Math. Soc.}}, 351(3):1151--1169, 1999.


\bibitem{CD03}
T.~{Coulhon} and X.~T. {Duong}.
\newblock {Riesz transform and related inequalities on noncompact Riemannian manifolds}.
\newblock {\em {Comm. Pure Appl. Math.}}, 56(12):1728–1751, 2003.




\bibitem{CJKS20}
T.~{Coulhon}, R.~{Jiang}, P.~{Koskela}, and A.~{Sikora}.
\newblock {Gradient estimates for heat kernels and harmonic functions}.
\newblock {\em {J. Funct. Anal.}}, 278(8):67, 2020.


\bibitem{Dev15}
B.~{Devyver}.
\newblock {A perturbation result for the Riesz transform}.
\newblock {\em {Ann. Sc. Norm. Super. Pisa Cl. Sci.(5)}},
\newblock 14(3): 937--964, 2015;

\bibitem{DR26}
B.~{Devyver}, E.~{Russ}.
\newblock {Reverse Inequality for Quasi-Riesz Transforms on Cable Systems.}
\newblock {\em {J. Geom. Anal.}}, 36:2, 2026.

\bibitem{DRY23}
B.~{Devyver}, E.~{Russ}, and M.~{Yang}.
\newblock {Gradient Estimate for the Heat Kernel on Some Fractal-Like Cable
	Systems and Quasi-Riesz Transforms}.
\newblock {\em Int. Math. Res. Not.},
2023(18):15537--15583, 

\bibitem{Fen26}
J.~{Feneuil},  
\newblock {In Spaces with a Slow Diffusion, the Riesz Transform is Unbounded on \(L^p, p\in(2,\infty)\)}
\newblock {\em {J. Geom. Anal.}}, 36:61, 2026.



\bibitem{GY26}
J.~{Gao} and M.~{Yang}.
\newblock{H\"older regularity of harmonic functions on metric measure spaces},
\newblock {\em {Adv. Math.}}, 488:110797, 2026.

\bibitem{GYZ26}
J.~{Gao}, Z.~{Yu} and J.~{Zhang}.
Heat kernel-based p-energy norms on metric measure spaces, to appear Math. Proc. Camb. Phil. Soc., 2026.

\bibitem{Gri92}
A.~{Grigor'yan}.
\newblock {The heat equation on noncompact Riemannian manifolds}.
\newblock {\em {Math. USSR, Sb.}}, 72(1), 1992.

\bibitem{GL20A}
Q.~{Gu} and K.-S.{Lau},
\newblock{ Dirichlet forms and convergence of {B}esov
	norms on self-similar sets},
\newblock {\em {Ann. Acad. Sci. Fenn. Math.}},
45:625--646, 2020.

\bibitem{GL20B}
Q.~{Gu} and K.-S.{Lau},
\newblock{ Dirichlet forms and critical exponents on fractals},
\newblock {\em {Trans. Amer. Math. Soc. }},
373(3):1619--1652, 2020.


\bibitem{HW06}
J.~{Hu} and X.~{Wang}.
\newblock{Domains of Dirichlet forms and effective resistance estimates on p.c.f. fractals}.
\newblock {\em { Studia Math.}}, 177(2):153--172, 2006.

\bibitem{LY86}
P.~{Li} and S.~T. {Yau}.
\newblock {On the parabolic kernel of the Schr\"odinger operator}.
\newblock {\em {Acta Math.}}, 156:154--201, 1986.


\bibitem{Jiang21}
R.~{Jiang}.
\newblock{Riesz transform via heat kernel and harmonic functions on non-compact manifolds}.
\newblock {\em {Adv. Math.},}
377:107464, 2021.

\bibitem{KS24}
N.~{Kajino} and R.~{Shimizu}.
\newblock {Contraction properties and differentiability of $p$-energy forms
	with applications to nonlinear potential theory on self-similar sets}.
 arXiv:2404.13668v2, 2024.

\bibitem{Kig89}
J.~{Kigami}.
\newblock {A harmonic calculus on the Sierpi\'nski spaces}.
\newblock {\em {Japan J. Appl. Math.}}, 6(2):259--290, 1989.

\bibitem{Kig93}
J.~{Kigami}.
\newblock {Harmonic calculus on p.c.f.\ self-similar sets}.
\newblock {\em {Trans. Amer. Math. Soc.}}, 335(2):721--755, 1993.


\bibitem{Kig01book}
J.~{Kigami}.
\newblock {\em Analysis on fractals}, volume 143 of {\em Cambridge Tracts in
	Mathematics}.
\newblock Cambridge University Press, Cambridge, 2001.

\bibitem{Kig12}
J.~{Kigami}.
\newblock Resistance forms, quasisymmetric maps and heat kernel estimates.
\newblock {\em Mem. Amer. Math. Soc.}, 216(1015):vi+132, 2012.

\bibitem{Sal90}
L.~{Saloff-Coste}.
\newblock {Analyse sur les groupes de Lie \`a croissance polyn\^omiale.
	(Analysis on Lie groups of polynomial growth)}.
\newblock {\em {Ark. Mat.}}, 28(2):315--331, 1990.

\bibitem{Sal92}
L.~{Saloff-Coste}.
\newblock {A note on Poincar\'e, Sobolev, and Harnack inequalities}.
\newblock {\em {Int. Math. Res. Not.}}, 1992(2):27--38, 1992.

\bibitem{Sal95}
L.~{Saloff-Coste}.
\newblock {Parabolic Harnack inequality for divergence form second order
	differential operators}.
\newblock {\em {Potential Anal.}}, 4(4):429--467, 1995.

\bibitem{Str83}
R.~S. {Strichartz}.
\newblock {Analysis of the Laplacian on a complete Riemannian manifold}.
\newblock {\em {J. Funct. Anal.}}, 52:48--79, 1983.

\bibitem{Str00}
R.~S. {Strichartz}.
\newblock {Taylor approximations on Sierpinski gasket type fractals}.
\newblock {\em {J. Funct. Anal.}}, 174(1):76--127, 2000.

\bibitem{Tep00}
A.~{Teplyaev}. 
\newblock {Gradients on fractals},
\newblock {\em {J. Funct. Anal.}},
	 174(1):128--154, 2000.

\bibitem{THP14}
D.~{Tang}, R.~{Hu}, and C.~{Pan}.
\newblock {H\"older estimates of harmonic functions on a class of p.c.f.
	self-similar sets}.
\newblock {\em {Anal. Theory Appl.}}, 30(3):296--305, 2014.

\end{thebibliography}

\end{document}